\documentclass[12pt]{article}

\usepackage{amsmath,amsfonts,amssymb,amsthm
}

\usepackage{graphicx}
\usepackage{hyperref}
\usepackage{color}

\usepackage[figurewithin=section]{caption}

\theoremstyle{definition}
\newtheorem{remark}{Remark}
\newtheorem{example}{Example}

\newtheoremstyle{mytheorem}{0.5cm}{0.2cm}{\slshape}{ }{\bfseries}{.}{ }{}
\theoremstyle{mytheorem}
\newtheorem{theorem}{Theorem}
\newtheorem{lemma}{Lemma}
\newtheorem{proposition}{Proposition}

\newtheorem{corollary}{Corollary}
\newtheorem{assumption}{Assumption}

\newcommand{\E}{\mathbf{E}}

\newcommand{\Vol}{\textnormal{Vol}}
\newcommand{\Var}{\textnormal{\textbf{Var}}}
 \renewcommand{\P}{\mathbf{P}}
 \newcommand{\cl}[1]{\textnormal{cl}(#1)}
 \newcommand{\ins}[1]{\textnormal{int}(#1)}
\newcommand{\X}{\chi}

\newcommand{\dK}{d_{\mathcal{K}}}

\newlength{\querylen}
\usepackage{fancybox}

\makeatletter
\@addtoreset{remark}{section}
\makeatother

\title{Boundary density and Voronoi set estimation for irregular sets}
\author{Rapha\"el Lachi\`eze-Rey\thanks{raphael.lachieze-rey@parisdescartes.fr, Laboratoire MAP5 (UMR CNRS 8145), Universit\'e Paris Descartes, Sorbonne Paris Cit\'e}~ and Sergio Vega \thanks{Laboratoire MAP5 (UMR CNRS 8145), Universit\'e Paris Descartes, Sorbonne Paris Cit\'e}}
\begin{document} \maketitle


\begin{center}\textbf{Abstract}
\end{center}
 
In this paper, we study the inner and outer boundary densities of some sets with self-similar boundary having Minkowski dimension $s>d-1$ in $\mathbb{R}^{d}$. These quantities turn out to be crucial in some problems of set estimation, as we show here for the Voronoi approximation of the set with a random input constituted by $n$ iid points in some larger bounded domain. We prove that some classes of such sets have positive inner and outer boundary density, and therefore satisfy Berry-Esseen bounds in $n^{-s/2d}$ for Kolmogorov distance. The Von Koch flake serves as an example, and a set with Cantor boundary as a counter-example. We also give the almost sure rate of convergence of Hausdorff distance between the set and its approximation.



\paragraph{Notations} We designate by $d(.,.)$ the Euclidean distance between points or subsets of $\mathbb{R}^{d}$. The closure, the interior, the topological boundary and the diameter of a set $A\subset\mathbb{R}^{d}$ are designated by $\cl{A}$, $\ins{A}$, $\partial A$, $\text{diam}(A)$ respectively. The open Euclidean ball with center $x$ and radius $r$ in $\mathbb{R}^{d}$ is noted $B(x,r)$.

Given two sets $A,B$, we write $A+B$ for $\{c\in \mathbb{R}^{d}\mid c=a+b, a\in A, b\in B \}$. The Hausdorff distance between $A$ and $B$ is designated by $d_{H}(A,B)$, that is
\begin{equation*}
d_{H}(A,B)=\inf\{r>0:\,A\subset B+B(0,r),\,B\subset A+B(0,r)\}.
\end{equation*}
$\Vol$ is the $d$-dimensional Lebesgue measure and $\kappa_{d}$ is the volume of the Euclidean unit ball. For $s>0,$ $\mathcal{H}^{s}$ is the $s$-dimensional Hausdorff measure on $\mathbb{R}^{d}$.

Throughout the paper, $K\subset \mathbb{R}^{d}$ is a non-empty compact set with positive volume. The letters  $c,C$ are reserved to indicate positive constants that depend only on fixed parameters like $K$ or $d$, and which value may change from line to line.

\section*{Background}

Set estimation theory is a topic of nonparametric statistics where an unknown set $K$ is estimated, based on partial random information.  The random input generally consists in a finite sample $\chi $ of points, either IID variables \cite{CFR,Pen07} or a Poisson point process \cite{HevRei,JimYuk,ReiSpoZap}. Based on the information of which of those points belong or not to $K$, one can reconstruct a random approximation $K_{\chi }$ of $K$ and study the asymptotic quality of the approximation. See the recent survey \cite[Chap. 11]{KenMol} about related works in nonparametric statistics.

The results generally require the set to be smooth in some sense.  In the literature, the set under study is assumed to be convex \cite{ReiSpoZap,Sch12},  $r$-convex \cite{CueRod04,Rod07}, to have volume polynomial expansion \cite{BCCF}, positive reach, or a $(d-1)$-rectifiable boundary \cite{JimYuk}. Another class of regularity assumptions usually needed is that of \emph{sliding ball} or \emph{rolling ball} conditions (\cite{CFP,Walth97,Walth99}). The most common form of this condition is that in every point $x$ of the boundary, there must be a ball touching $x$ and contained either in $K$, in $K^{c}$, or  both. 

In those works, the random approximation model $K_{\chi }$ can be the union of balls centred in the points of $\chi $ with well tuned radius going to $0$,  a level set of the sum of appropriately scaled kernels centred on the random points, or else. Recently, a different model has been used in stochastic geometry, based on the Voronoi tessellation associated with $\chi $. One defines $K_{\chi }$ as the union of all Voronoi cells which centers lie in $K$, assuming  that points of $\chi $ fall indifferently inside and outside $K $, as $K$ is unknown. This is equivalent to defining $K_{\chi }$ as the set of points that are closer to $\chi \cap K$ than to $\chi \cap K^{c}$.

This elegant model presents practical advantages in set estimation. For volume estimation the bias and standard deviation rates of the Voronoi approximation seem to be best among all estimators of which the authors are aware of, and hold under almost no assumption on $K$. Regarding shape estimation, Voronoi approximation also consistently estimates $K$ and $\partial K$ in the sense of the Hausdorff distance (Proposition~\ref{consistency-hausdorff}), and here again convergence rates and necessary assumptions compare favourably to those of other estimators (see Theorem~\ref{theorem-haus} and the following Remarks).

An heuristic explanation of these features is that the estimator naturally fills in regions inside $K$ where the sample $\chi$ is sparse, without need for convexity-like assumptions on $K$ \cite{Rod07} or parameter tuning \cite{BiauCadre09,CFR,DevWise}.

The reader will find a more formal presentation of Voronoi approximation along with a summary of existing results \cite{CalChe13,HevRei, JimYuk, ReiSpoZap, Sch12} in Section~\ref{sec:voronoi}.

\section*{Approach and main results}

This work was inspired and is closely related to \cite{LacPec15}, in which a central limit theorem and variance asymptotics for $\Vol(K_{\chi})$ were obtained for binomial input under very weak assumptions on $K$. Here we slighlty enhance their central limit theorem by showing that $\Vol(K_{\chi})$ can be recentered by $\Vol(K)$ instead of $\E(\Vol(K_{\chi}))$. Explicitly, for suitable $K$, we have for each $\varepsilon >0$ a constant $C_{\varepsilon >0}$ such that
\begin{equation}
\label{eq:berry-esseen-2-intro}
\sup_{t\in \mathbb{R}}\left|  \P \left( \frac{\Vol (K_{\chi _{n}})-\Vol(K)}{\sqrt{ \Var(\Vol(K_{\chi _{n}}))}}\geqslant t  \right)-\P(N\geqslant t) \right| \leqslant C_{\varepsilon }n^{-s/2d}\log(n)^{4-s/d+\varepsilon }
\end{equation}
where $s$ is the Minkowski dimension of $\partial K$ (see Section~\ref{sec:MinkoContents}).

We also show that with Poisson input we have the almost sure convergence rates for the Hausdorff distance
\begin{align}
& c \leqslant \liminf\limits_{n\rightarrow +\infty} \frac{d_{H}(K,K_{\chi_{n}})}{(n^{-1} \ln(n))^{1/d}} \leqslant \limsup\limits_{n\rightarrow +\infty} \frac{d_{H}(K,K_{\chi_{n}})}{(n^{-1} \ln(n))^{1/d}} \leqslant C , \label{eq:haus-intro} \\
& c \leqslant \liminf\limits_{n\rightarrow +\infty} \frac{d_{H}(\partial K, \partial K_{\chi_{n}})}{(n^{-1} \ln(n))^{1/d}} \leqslant \limsup\limits_{n\rightarrow +\infty} \frac{d_{H}(\partial K,\partial K_{\chi_{n}})}{(n^{-1} \ln(n))^{1/d}} \leqslant C, \label{eq:haus-intro-2}
\end{align}
thus answering a query raised in \cite{HevRei} and extending the results obtained in \cite{CalChe13}.


\paragraph{}

The assumptions on $K$ necessary for \eqref{eq:berry-esseen-2-intro},\eqref{eq:haus-intro} and \eqref{eq:haus-intro-2} to hold are worth of interest on their own. They are broad enough to allow for irregular $K$, a feature which few estimators possess and is useful in some applications (see \cite{CFR,JimYuk}, and references therein). Also, they are not specific to Voronoi approximation, and might be crucial for other estimators. They are mainly concerned with the densities of $K$ at radius $r$ in $x$, defined by
\begin{align*}
f^{K}_{r}(x)=& \enspace \frac{\Vol(K\cap B(x,r))}{\Vol(B(x,r))}, \\
f^{K^{c}}_{r}(x)=& \enspace \frac{\Vol(K^{c}\cap B(x,r))}{\Vol(B(x,r))}.
\end{align*}
For ease of notation, we shall simply write $f_{r}$ for $f^{K}_{r}$ and $g_{r}$ for $f_{r}^{K^{c}}$, $K$ being implicit in all of the paper. Boundary densities have already appeared in set estimation theory \cite{Cue90, CalChe13, CFR}, where a set $K$ is said to be \emph{standard} whenever $f_{r}\geqslant \varepsilon$ on $K$ for some fixed $\varepsilon > 0$ and all small enough $r$. Here we shall prefer to specify \emph{inner standard} since we are also interested in cases where the inequality $g_{r}>\varepsilon$ holds. In the latter case, $K$ is said to be \emph{outer standard}, and if $K$ is both inner and outer standard $K$ will be said to be \emph{bi-standard}. The condition on $K$ for \eqref{eq:haus-intro-2} to hold is essentially bi-standardness, which is a usual assumption in set estimation \cite[Theorem 1]{CFR}.

The requirement for \eqref{eq:berry-esseen-2-intro} to hold seems to be new in set estimation theory. {It consists in a positive $\liminf$ {as $r\rightarrow 0$} of the quantity
\begin{equation*}
\frac{1}{\Vol(\partial K_{r})}\int_{K^{c}} f_{r} =\frac{1}{\Vol(\partial K_{r})} \int_{K}g_{r},
\end{equation*} 
where $\partial K_{r}$ is the set of points within distance $r$ from $\partial K$}. See Assumption~\ref{roll} and Proposition~\ref{roll-bis} for precise statements {and equivalent assertions}. The above quantity measures the interpenetration of $K$ and $K^{c}$ along their common boundary, since the greater it is, the more homogenously $K$ and $K^{c}$ are distributed along $\partial K$. This lead us to name the condition of Assumption~\ref{roll} the \emph{boundary permeability condition}.

\paragraph{}
Study of densities on the boundary is also related with works in geometric measure theory. Points for which $\lim f_{r}(x)$ is $0$ or $1$ are considered resp. as the measure-theoretic exterior and interior of $K$, while other points constitute $\partial ^{*}K$ the \emph{essential boundary} of $K$. Federer \cite[Th.3.61]{AFP} proved that if $K$ is a measurable set with finite measure-theoretic perimeter then $f_{r}\rightarrow 1/2$ on most of the essential boundary.

We address here the question of whether comparable results hold if $\partial K$ is an irregular set, with self-similar features. In general, such boundaries have a Hausdorff dimension $s>d-1$  and don't have finite perimeter. But, because of self-similarity, the densities $f_{r},g_{r}$ should nevertheless have continuous and somehow periodical fluctuations in $r$, and therefore a positive infimum. This is confirmed by Theorem~\ref{theorem1}, which gives, for $K$ with self-similar boundary, a set of conditions under which  $f_{r}>\varepsilon $ on the boundary uniformly in $r>0$. It is even proved that a ball with radius $cr$ for some $c>0$ can be rolled inside or outside the boundary, staying within a distance $r$ from the boundary, but not touching it (otherwise self-similar boundaries would be excluded). Theorem \ref{theorem1} applies for instance to the Von Koch flake in dimension $2$, which is therefore well-behaved under Voronoi approximation and satisfies \eqref{eq:berry-esseen-2-intro}, \eqref{eq:haus-intro} and \eqref{eq:haus-intro-2}.

Some sets with self-similar boundary do not fall under the scope of this result, and we also give example of a self-similar set $K_{\text{cantor}}$ with Cantor-like self-similar boundary not satisfying the boundary permeability condition. Simulations we ran suggest that this irregularity of $K_{\text{cantor}}$'s boundary indeed  reflects on the behaviour of its Voronoi approximation and prevents the variance of the estimator from satisfying an asymptotic power law like in \eqref{eq:var}. This suggests that the boundary permeability condition is indeed significant in set estimation and not merely a contingent constraint due to the methods used to obtain \eqref{eq:berry-esseen-2-intro}.

\section*{Plan}

The plan of the paper is as follows. In Section \ref{sec:self-similar}, we recall basic facts and definitions about self-similar sets, especially regarding upper and lower Minkowski contents. We then give conditions under which sets with self-similar boundaries are standard. Voronoi approximation is formally introduced in Section \ref{sec:voronoi}. We then derive the volume normal approximation for sets with well-behaved boundaries, {as well as} Hausdorff distance results. We also develop the counter example $K_{\text{cantor}}$ that satisfies neither the hypotheses of Theorem \ref{theorem1} nor the volume approximation variance asymptotics \eqref{eq:var}.

\section{Self-similar sets}
\label{sec:self-similar}

\subsection{Self-similar set theory}
This subsection contains a review of some classic results of self-similar set theory. A more precise treatment of the subject and most of the results stated here can be found in \cite{Fal}. Broadly speaking, a set is self-similar when arbitrarily small copies of the set can be found in the neighbourhood of any of its points. This suggests that a self-similar set should be associated with a family of similitudes.

\paragraph{}
Let $\{\phi_{i}, i \in I\}$ be a finite set of contracting similitudes. Such a set is called an \emph{iterated function system}. Define the following set transformation
\begin{align*}
\psi: \mathcal{P}(\mathbb{R}^{d})&\longrightarrow \enspace \mathcal{P}(\mathbb{R}^{d}) \\
E \quad &\longmapsto \enspace \bigcup_{i} \phi_{i}(E).
\end{align*}
It is easily seen that $\psi$ is contracting for the Hausdorff metric, which happens to be complete on $\mathcal{K}^{d}$, the class of non-empty compact sets of $\mathbb{R}^{d}$. By a fixed point theorem, there is an unique set $E \in \mathcal{K}^{d}$ satisfying $\psi(E)=E$, which is by definition the self-similar set associated with the $\phi_{i}$.

\paragraph{} If there is a bounded open set $U$ such as $\psi(U)=\bigcup\phi_{i}(U)\subset U$ with the union disjoint, then necessarily $E \subset \cl{U}$ and the $\phi_{i}$ are said to satisfy the \emph{open set condition}. Schief proved in \cite{Sch94} that we can pick $U$ so that $U \cap E$ is not empty. {This stronger assumption is referred to as the \emph{strong open set condition} in the literature}.

\paragraph{}
The similarity dimension of $E$ is the unique $s$ satisfying 
\begin{equation*}
\sum \lambda_{i}^{s}=1
\end{equation*}
where $\lambda_{i}$ is the stretching factor of $\phi_{i}$. When the open set condition holds, this similarity dimension is also the Hausdorff dimension and the Minkowski dimension of $E$. Furthermore, $E$'s upper and lower $s$-dimensional Minkowski contents ({see Subsection~\ref{sec:MinkoContents}}) are finite and positive. This is an easy and probably known result, but since we have not found it explicitly stated and separately proven in the literature, we will do so here in Proposition~\ref{lemma1}. We will need the following classical lemmae, that we prove for completeness.

\begin{lemma}
\label{lemma1}
Let $(U_{i})$ be a collection of disjoint open sets in $\mathbb{R}^{d}$ such that each $U_{i}$ contains a ball of radius $c_{1}r$ and is contained in a ball of radius $c_{2}r$. Then any ball of radius $r$ intersects at most $(1+2c_{2})^{d}c_{1}^{-d}$ of the sets $\cl{U_{i}}$.

\begin{proof}
Let $B$ be a ball of center $x$ and radius $r$. If some $\cl{U_{i}}$ intersects $B$ then $\cl{U_{i}}$ is contained in the ball $B'$ of center $x$ and radius $r(1+2c_{2})$. If $q$ different $\cl{U_{i}}$ intersect $B$ then there are $q$ disjoint balls of radius $c_{1}r$ inside $B'$, and by comparing volumes $q\leqslant (1+2c_{2})^{d}c_{1}^{-d}$.
\end{proof}
\end{lemma}

\begin{lemma}
\label{lemma2}
Suppose that $E$ and the $\phi_{i}$ satisfy the open set condition with  $U$. Then for every $r < 1$ we can find a finite set $\mathcal{A}$ of similarities $\Phi_{k}$ with ratios $\Lambda_{k}$ such that 
\begin{enumerate}
\item The $\Phi_{k}$ are composites of the $\phi_{i}$.
\item The $\Phi_{k}(E)$ cover $E$.
\item The $\Phi_{k}(U)$ are disjoint.
\item $\displaystyle\sum \Lambda_{k}^{s} =1$ where $s$ is the similarity dimension of $E$.
\item $\min_{i}(\lambda_{i})r \leqslant \Lambda_{k} < r$ for all $k$.
\end{enumerate}
 
\begin{proof}
We give an algorithmic proof. Initialise at step $0$ with $\mathcal{A}=\{Id\}$. At step $n$ replace every $\Phi \in \mathcal{A}$ with ratio greater than $r$ by the similarities $\Phi\circ\phi_{i}, i \in I$. Stop when the process becomes stationary, which will happen no later than step $\lceil\ln(r)/\ln({\max(\lambda_{i})})\rceil$.

\paragraph{} Obviously, point 1 is satisfied. We will prove the next three points by induction. At step $0$, all of $E$ is covered by the $\Phi_{k}(E)$, the $\Phi_{k}(U)$ are disjoint, and the $\Lambda_{k}^{s}$ sum up to $1$. The first property is preserved when $\Phi$ is replaced by the $\Phi\circ\phi_{i}$, since $\Phi(E)=\Phi(\psi(E))=\bigcup \Phi\circ\phi_{i}(E)$. Likewise, the $\Phi\circ\phi_{i}(U)$ are disjoint one from each other because $\Phi$ is one-to-one, and disjoint from the other $\Phi_{k}(U)$ because $\bigcup \Phi\circ\phi_{i}(U)= \Phi(\psi(U)) \subset \Phi(U)$, which yields point 3. For point 4 note that if $\Phi$ has ratio $\Lambda$, then the $\Phi\circ\phi_{i}$ have ratios $\Lambda\lambda_{i}$ and $\Lambda^{s}=\Lambda^{s}\sum\lambda_{i}^{s}=\sum(\Lambda\lambda_{i})^{s}$ so the sum of the $\Lambda_{k}^{s}$ remains unchanged by the substitution. Finally, since $r < 1$, every final set of the process has an ancestor with ratio greater than $r$. This gives the lower bound for point 5; the upper bound comes from the fact that the process ends.
\end{proof}
\end{lemma}

\begin{remark}
The process in the proof of Lemma~\ref{lemma2} is often resumed by the formula 
\begin{equation*}
\mathcal{A}=\{\phi_{i_{1}}\circ\phi_{i_{2}}\ldots \circ\phi_{i_{n}} \mid \prod_{k=1}^{n} \lambda_{i_{k}} < r \leqslant \prod_{k=1}^{n-1} \lambda_{i_{k}}\}.
\end{equation*}
\end{remark}

\subsection{Minkowski contents of self-similar sets}
\label{sec:MinkoContents}
Recall that the $s$-dimensional lower Minkowski content of a non-empty bounded set $E\subset \mathbb{R}^{d}$ can be defined as
\begin{equation*}
\liminf_{r>0}\frac{\Vol(E+B(0,r))}{r^{d-s}}.
\end{equation*}
Similarly, the $s$-dimensional upper Minkowski content of $E$ is 
\begin{equation*}
\limsup_{r>0}\frac{\Vol(E+B(0,r))}{r^{d-s}}.
\end{equation*}
In this paper, when both contents are finite and positive, we will simply say that $E$ has upper and lower Minkowski contents. This leaves no ambiguity on the choice of $s$, since in that case $s$ is necessarily the Minkowski dimension of $E$, i.e 
\begin{equation*}
s= d-\lim\limits_{r\rightarrow 0}\frac{\ln(\Vol(E+B(0,r)))}{\ln(r)}.
\end{equation*}
We show below that self-similar sets always have upper and lower Minkowski contents. One can find an alternative proof for the lower content in \cite[Paragraph 2.4]{Gat00}, it can also be considered a consequence of $\mathcal{H}^{s}(E)>0$, like suggested in \cite{Mattila}.
\begin{proposition}
\label{prop1}
Let $E$ be a self-similar set satisfying the open set condition with similarity dimension $s$. Then $E$ has finite positive $s$-dimensional upper and lower Minkowski contents.

\begin{proof}
As before, let $\phi_{i}$ be the generating similarities of $E$, $\lambda_{i}$ their ratios, $\psi:A\mapsto \bigcup \phi_{i}(A)$ the associated set transformation, and $U$ the open set of the open set condition. Choose any $0 < r < 1$ and define the $\Phi_{k},\Lambda_{k}$ as in Lemma~\ref{lemma2}. Finally, write $E_{k}=\Phi_{k}(E), U_{k}=\Phi_{k}(U)$.

\paragraph{}
We approximate $E+B(0,r)$ by the sets $E_{k}+B(0,r)$, who are similar to the $E+\Phi_{k}^{-1}(B(0,r))$. By construction $\Phi_{k}^{-1}(B(0,r))$ is a ball with a radius belonging to $[1,(\min_{i}\lambda_{i})^{-1}]$, so that$$\Vol(B(0,1))\leqslant \Vol(E+\Phi_{k}^{-1}(B(0,r))) \leqslant \Vol(B(0,\text{diam}(E)+(\min_{i}\lambda_{i})^{-1})),$$ because $E$ is not empty. Applying $\Phi _{k}$ gives
\begin{equation*}
c\Lambda_{k}^{d}\leqslant \Vol(E_{k}+B(0,r)) \leqslant C\Lambda_{k}^{d}.
\end{equation*}

Since $E+B(0,r) \subset \bigcup_{k} E_{k}+B(0,r)$ and $\sum\Lambda _{k}^{s}=1$ we immediately get the upper bound
\begin{align*}
\Vol(E+B(0,r)) \leqslant &\enspace \sum \Vol(E_{k}+B(0,r)) \\
\leqslant &\enspace \sum C\Lambda_{k}^{d} \\
\leqslant &\enspace C\sum \Lambda_{k}^{s}r^{d-s} \\
\leqslant &\enspace Cr^{d-s}.
\end{align*}

\paragraph{}
For the lower bound we apply Lemma~\ref{lemma1} to the disjoint $U_{k}$. Since $U$ is open we can put some ball of radius $c_{1}$ in $U$, and conversely we can put $U$ in some ball of radius $c_{2}$, since $U$ is bounded. This means that each of the $U_{k}$ contains a ball of radius $ r\min_{i}(\lambda_{i})c_{1}\leqslant \Lambda_{k}c_{1}$ and is contained in a ball of radius $rc_{2}\geqslant\Lambda_{k}c_{2}$. So for any $x\in E+B(0,r)$, $B(x,r)$ intersects at most $q$ of the $E_{k}$ (since $E_{k}\subset\cl{U_{k}}$) with $q$ a positive integer independent of $r$ and $x$. This can be rewritten $\displaystyle\mathbf{1}_{E+B(0,r)}\geqslant \frac{1}{q}\sum\mathbf{1}_{E_{k}+B(0,r)}$. Integrating we get $\displaystyle \Vol(E+B(0,r))\geqslant \frac{1}{q} \sum \Vol(E_{k}+B(0,r))$ so that

\begin{align*}
\Vol(E+B(0,r)) \geqslant &\enspace\frac{1}{q} \sum \Vol(E_{k}+B(0,r)) \\
\geqslant &\enspace\frac{c}{q} \sum \Lambda_{k}^d \\
\geqslant &\enspace\frac{c}{q}(\min_{i}\lambda_{i})^{d-s} \sum \Lambda_{k}^{s}r^{d-s} \\
\geqslant &\enspace cr^{d-s}. \\
\end{align*}
\end{proof}
\end{proposition}

\subsection{Boundary regularity}
\label{sec:boundary-density}

In order to formulate our result, we  introduce the notion of \emph{proper} and \emph{improper} points. A point $x\in \mathbb{R}^{d}$ is proper to $K$ if $\Vol(O\cap K) > 0$ for every neighbourhood $O$ of $x$, it is improper to $K$ otherwise. In other words, the set $K^{\text{prop}}$ of proper points of $K$ is the support of the measure $\Vol(K\cap \cdot)$. Further use of proper points will be made in Section~\ref{essential2}. We can already note that $K$ must have no improper points if we want a positive lower bound for the $f_{r}$ on $K$. 

Our result holds for self-similar subsets $E$ of $\partial K$ satisfying the following assumption:

\begin{assumption}
\label{ass:1}
A {self-similar} subset $E$ of $\partial K$ satisfies the strong open set condition with some set $U$ such that $U\cap \partial K\subset E$ and  $U \setminus \partial K$ has finitely many connected components.
\end{assumption}

This assumption can be justified heuristically: if $E$ cuts its neighbourhood into infinitely many connected components, then because of self-similarity it also does so locally, and $K$ and $K^{c}$ are too disconnected to contain the balls mentioned in Theorem~\ref{theorem1}. Example~\ref{cantor-example} will show that these concerns are legitimate.

\begin{theorem}
\label{theorem1}
Let $K$ be a non-empty compact set with no improper points and $\Vol(\partial K)=0$. Let $E$ be a self-similar subset of $\partial K$  for which Assumption \ref{ass:1} holds. Then there are constants $\delta, \varepsilon >0$ such that, for all $r< \delta, x\in E$, both $B(x,r)\cap K^{c}$ and $B(x,r)\cap K$ contain a ball of radius $\varepsilon r$.
\end{theorem}

\begin{proof}
Let $\phi_{i}$ be the generating similarities of $E$, $\lambda_{i}$ their ratios, $\psi:A\mapsto \bigcup \phi_{i}(A)$ the associated set transformation. {Denote by} $V_{j}$ the connected components of $U\setminus E$. Since there are finitely many of them, we can suppose they all contain a ball of radius $\tau>0$. Fix any $0 < r < 1$ and $x \in E$.

\paragraph{} Lemma~\ref{lemma2} shows that there is a similarity $\Phi$ with ratio $\Lambda$ such that $\min(\lambda_{i})r \leqslant \Lambda < r$ and $x \in \Phi(E)$. It follows that $\Phi(U) \subset B(x,r)$.  We also have $\Phi(U)\cap \partial K = \Phi(U)\cap E = \Phi(U\cap E)$. Indeed, for any point $x'$ of $E$ outside $\Phi(E)$ there is another similarity $\Phi'$ of Lemma~\ref{lemma2} such that $x'\in\Phi'(E)$ and $\Phi'(U)\cap\Phi(U)=\emptyset$, which implies $\cl{\Phi'(U)}\cap \Phi(U)=\Phi'(\cl{U})\cap\Phi(U)=\Phi'(E)\cap \Phi(U)=\emptyset$ so that $x' \notin \Phi(U)$.

\paragraph{}Consequently, for all $j$, $\Phi(V_{j})$ has no intersection with $\partial K$. So $\Phi(V_{j})\cap \ins{K}$ and $\Phi(V_{j})\cap K^{c}$ are two disjoint open set sets who cover $\Phi(V_{j})$, and we must have either $\Phi(V_{j}) \subset K$ or $\Phi(V_{j}) \subset K^{c}$.

\paragraph{}Since there is a point $y$ in $\Phi(U)\cap E$ and $K$ has no improper points, we must have $\Vol(K\cap U), \Vol(K^{c}\cap U) > 0$. Because $\Vol(E)=0$, this can only happen if one of the $\Phi(V_{j})$ is included in $K^{c}$ and another in $K$. Hence $B(x,r)\cap K$, $B(x,r)\cap K^{c}$ each contain a ball of radius $\Lambda\tau$. Since $\Lambda\geqslant \min(\lambda_{i})r$, the conclusion of the theorem holds with $\varepsilon=\min(\lambda_{i})\tau$.
\end{proof}

\begin{remark}
The theorem implies that $K,K^{c}$ have lower density bounds on $E$. More precisely, for appropriate $\delta, \varepsilon > 0$
\begin{align}
\label{eq:positive-density-boundary}
\forall  x\in E, r < \delta, \quad f_{r}(x), g_{r}(x)  \geqslant \varepsilon.
\end{align}
This weaker statement is enough for our purposes regarding Voronoi approximation.
\end{remark}

\paragraph{}
We show below that the Von Koch flake provides a concrete example of an irregular set satisfying the hypotheses of Theorem \ref{theorem1}.
\begin{example}
\label{flake}

Let $E$ be the self-similar set associated with the direct similarities $\phi_{i}:\mathbb{R}^{2} \rightarrow \mathbb{R}^{2}$ sending $S=A_{0}A_{4}$ to $a_{i}=A_{i-1}A_{i}$, for $i=1,2,3,4$, in the configuration of Figure~\ref{VKcurve}. Such sets $E$ are called Von Koch curves. Looking at the iterates $\psi^{(n)}(S)$ in Figure~\ref{VKiterates} gives an idea of the general form of the Von Koch curve and of why it is said to be self-similar.
\begin{figure}[!h]
\centering
\includegraphics[scale=0.8]{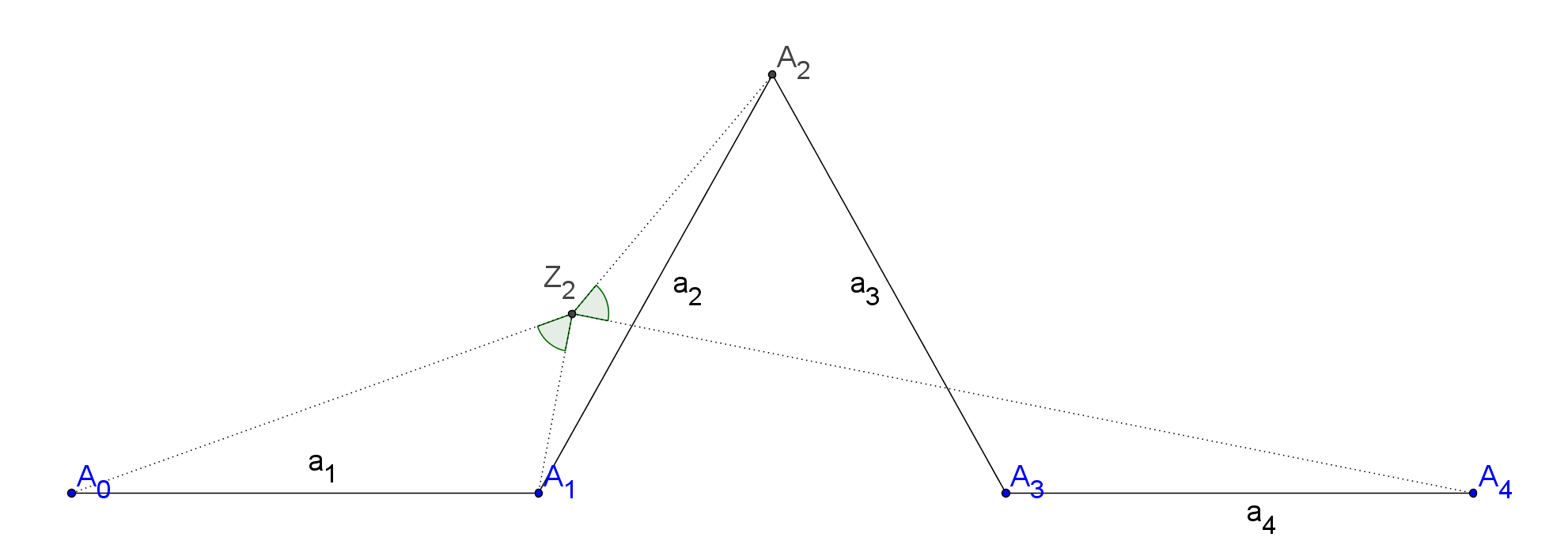}
\caption{\label{VKcurve}The generating similitudes of the Von Koch curve. $Z_{2}$ is the center of the similarity $\phi_{2}$.}
\end{figure}

\begin{figure}[!h]
\centering
\includegraphics[scale=0.5]{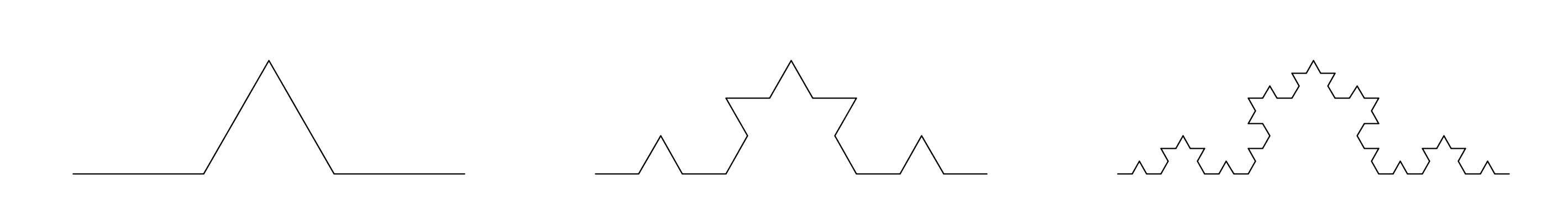}
\caption{\label{VKiterates}The sets $\psi^{(1)}(S),\psi^{(2)}(S),\psi^{(3)}(S)$.}
\end{figure}

\paragraph{}
Note that the $\psi^{(n)}(S)$ are curves, i.e the images of continuous mappings $\gamma_{n}:[0,1]\rightarrow \mathbb{R}^{2}$. The $\gamma_{n}$ can be chosen to be a Cauchy sequence for the uniform distance between curves in $\mathbb{R}^{2}$. Hence their limit $\gamma$ is also a continuous mapping, $\gamma([0,1])$ is compact and has distance $0$ with $E$ in the Hausdorff metric, so $\gamma([0,1])=E$ which proves that the Van Koch curve is, indeed, a curve. It can also be shown to be a non-intersecting curve (the image of an injective continuous mapping from $[0,1]$ into $\mathbb{R}^{2}$).

\paragraph{}
With a similar reasoning, if we stick three Von Koch curves of same size as in Figure~\ref{VKflake}, we get a closed non-intersecting curve $\mathcal{C}$. Jordan's curve theorem says $\mathbb{R}^{2}\setminus \mathcal{C}$ has exactly two connected components who both have $\mathcal{C}$ as topological boundary. The closure $K$ of the bounded component is a compact set with no improper points satisfying $\partial K =\mathcal{C}$. $K$ is called a Von Koch flake.
\begin{figure}[!h]
\centering
\includegraphics[scale=0.5]{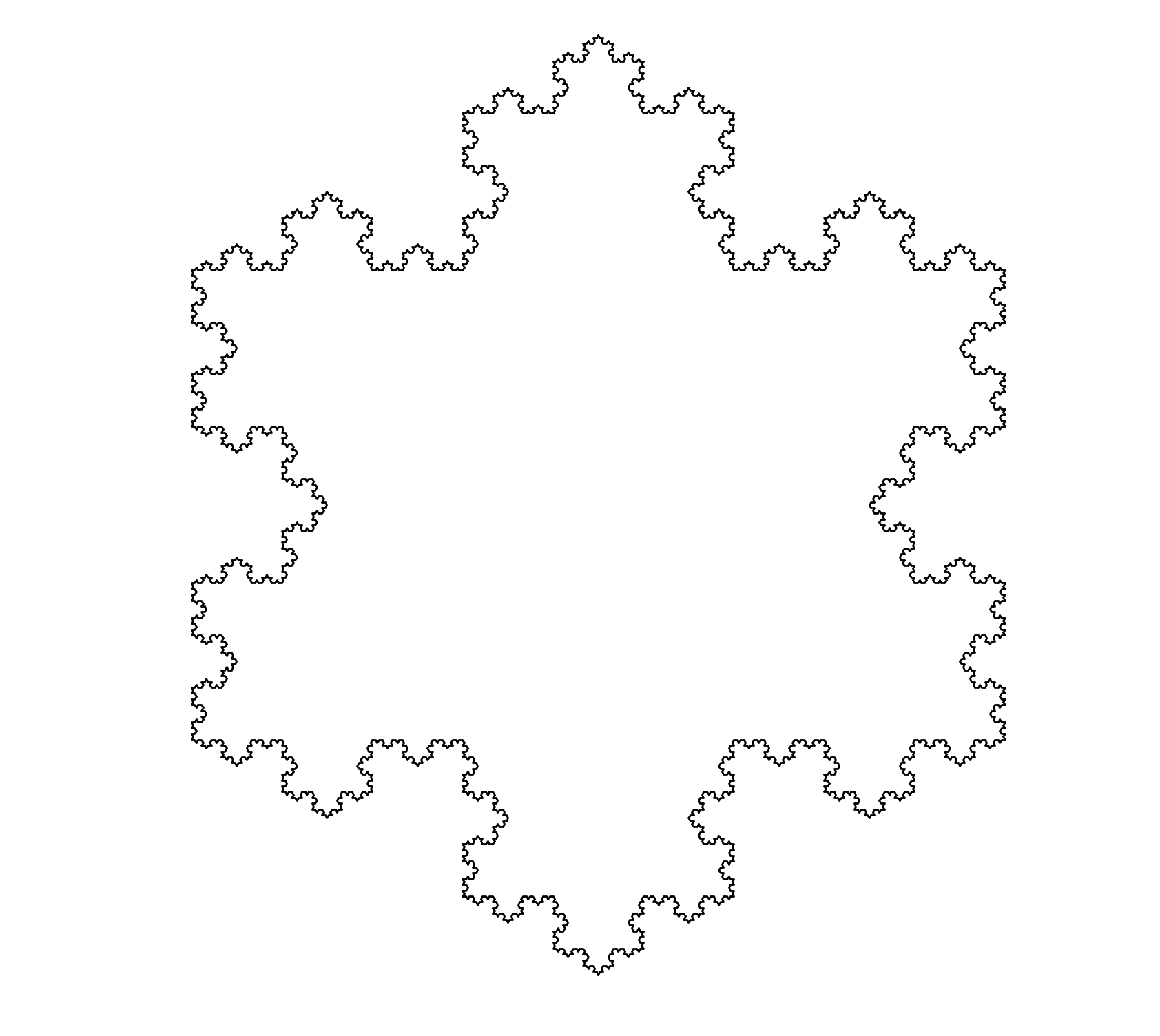}
\caption{\label{VKflake} The boundary of the Von Koch flake $K$.}
\end{figure}

\paragraph{}
Now, construct kites $C_{1},C_{2},C_{3}$ on each of the Von Koch curves $E_{1},E_{2},E_{3}$ making $\partial K$ as in Figure~\ref{VKkites}. It is easy to see that as long as the two equal angles of the lower triangle are flat enough, $C_{i}\cap\partial K=C_{i}\cap E_{i}$. Furthermore, applying Jordan's curve theorem to the $E_{i}$ and the two upper (resp. lower) edges of the corresponding $C_{i}$ shows that the $C_{i}\setminus E_{i}$ have exactly two connected components. The strong open set condition is also satisfied, so Theorem~\ref{theorem1} can be applied three times to obtain lower bounds for $f_{r}$ and $g_{r}$ on $\partial K$.

\begin{figure}[!h]
\centering
\includegraphics[scale=0.5]{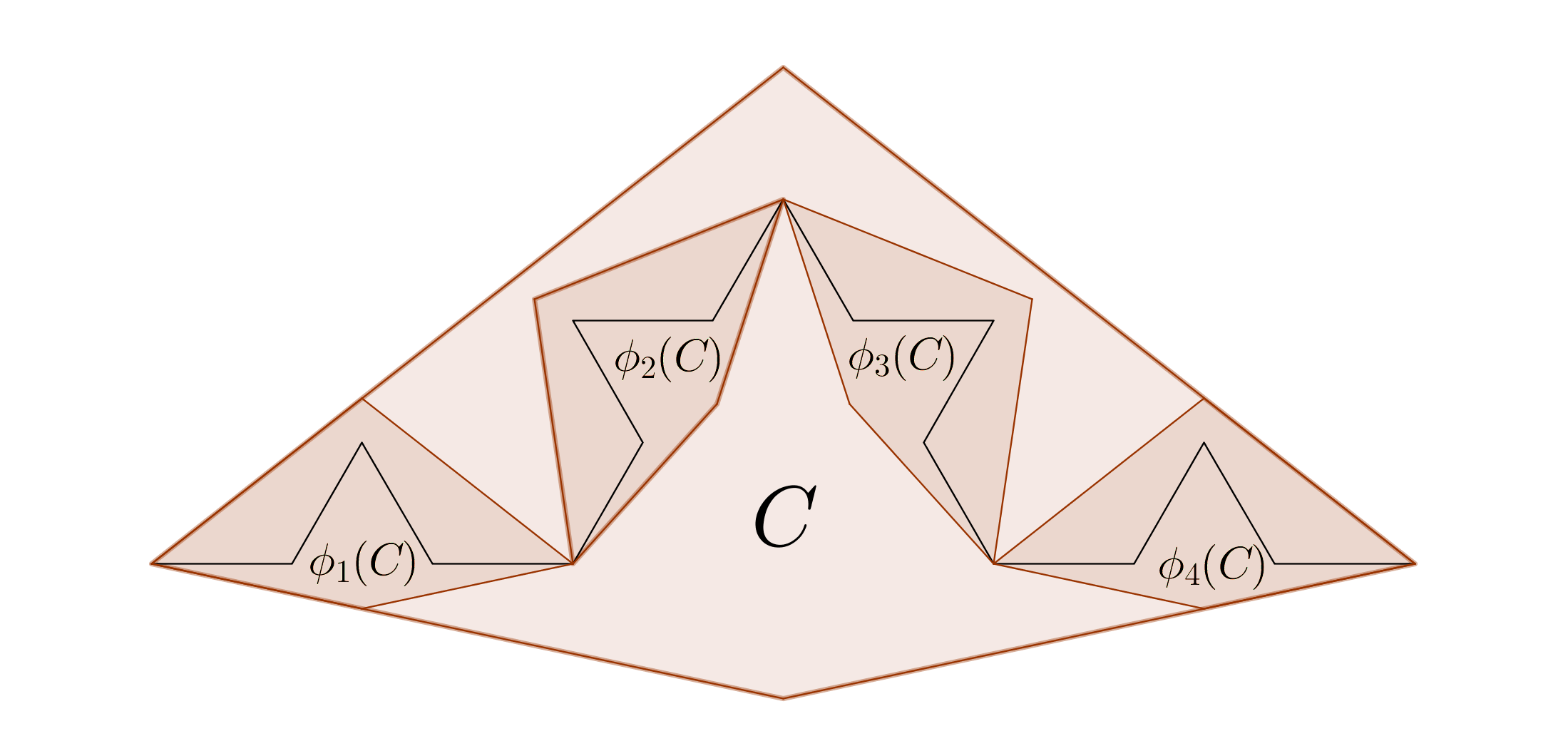}
\caption{\label{VKkites} Assumption~\ref{ass:1} is satisfied with the kite $C$.}
\end{figure}
\end{example}

Theorem~\ref{theorem1} is only concerned with the behaviour of $f_{r}$ and $g_{r}$ on $\partial K$, whereas standardness assumption require lower bounds on all of $K$ and $K^{c}$ respectively. The following lemma takes care of this issue.
\begin{lemma}
\label{lem:inside}
If for all $r<\delta$ we have $f_{r} \geqslant \varepsilon$ on $\partial K$, then for all $r<\delta$ we have $f_{r} \geqslant 2^{-d}\varepsilon$ on $K$. The same result holds if $f_{r}$ is replaced by $g_{r}$ and $K$ by $K^{c}$.
\begin{proof}
If $x$ is in $\partial K_{r/2}$ then $B(x,r)$ contains a ball of radius $r/2$ centered on $x'\in\partial K$, so $f_{r}(x)=\Vol(K\cap B(x,r))\kappa_{d}^{-1}r^{-d} \geqslant \Vol(K\cap B(x',r/2))\kappa_{d}^{-1}r^{-d} \geqslant \varepsilon 2^{-d}$. If $x$ is in $K\setminus \partial K_{r/2}$ then the ball $B(x,r/2)$ is contained in $K$ so that $f_{r}(x) \geqslant 2^{-d} \geqslant \varepsilon 2^{-d}$. So in all cases, if $x\in K$ then $f_{r}(x)\geqslant \varepsilon 2^{-d}$. Replacing $K$ by $K^{c}$ gives the result regarding $g_{r}$.
\end{proof}
\end{lemma}

\section{Voronoi approximation}
\label{sec:voronoi}

In this section, $\chi$ is a locally finite point process {, and $n\geqslant 1$}. If $\chi=\chi_{n} =\{X_{1},X_{2},\ldots,X_{n}\}$, where the $X_{i}$ are iid random points uniformly distributed over $[0,1]^{d}$, we speak of binomial input; if $\chi=\chi'_{\lambda }$ is a homogenous Poisson point process of intensity $\lambda >0$ we speak of Poisson input.

Define the Voronoi cell $\upsilon_{\chi}(x)$  of nucleus $x$ with respect to $\chi$ as the closed set of points closer to $x$ than to $\chi$ 
\begin{equation*}
\upsilon_{\chi}(x)=\{y \in \mathbb{R}^{d} : \forall x' \in \chi, d(x,y)\leqslant d(x',y)\}.
\end{equation*}
\paragraph{}
The Voronoi approximation $K_{\chi}$ of $K$ is the closed set of all points which are closer to $K\cap \chi$ than to $K^{c} \cap \chi$. Its name comes from the relation
\begin{equation*}
K_{\chi}=\bigcup_{x \in \chi \cap K} \upsilon_{\chi}(x).
\end{equation*}

The volume $\varphi (\chi )=\Vol(K_{\chi })$ first arised in \cite{KhmTor} as discriminating statistics in the two-sample problem. These authors proved a strong law of large numbers in dimension $1$ for the volume approximation. Explicit rates of convergence in higher dimensions were obtained by Reitzner and Heveling \cite{HevRei}, who proved that if $K$ is convex and compact and $\X=\X'_{\lambda}$ then
\begin{align*}
\E \varphi(\X)= & \enspace \Vol(K), \\
\Var(\varphi(\X))\leqslant & \enspace  C\lambda ^{-1-1/d}S(K), \\
\end{align*} where $S(K)$ is the surface area of $K$, all constants can be made explicit and depend only on $d$. They also studied the quantity $ \varphi _{\text{Per}}(\chi)=\Vol(K\Delta K_{\chi })$ to estimate the perimeter, after suitable renormalisation. Reitzner, Spodarev and Zaporozhets \cite{ReiSpoZap}  extended these results to sets  with finite variational perimeter, and also gave upper bounds for $\E|\varphi(\X'_{\lambda})^{q}-\Vol(K)^{q} | $ for $q\geq 1$. Schulte \cite{Sch12} obtained a matching lower bound for the variance with convex $K$, i.e. $cS(K)\lambda ^{-1-1/d}\leq \Var(\varphi(\X))$, and derived the corresponding CLT
\begin{align*}
\frac{\varphi(\X)-\E \varphi(\X)}{\sqrt{\Var(\varphi  (\X)})} \overset{(d)}{\longrightarrow} N.
\end{align*}
Very recently, Yukich \cite{Yuk15} gave quantitative Berry-Esseen bounds for this CLT similar to the ones that are stated here for binomial input.

When dealing with binomial input, which has been less studied than Poisson input, it is necessary to assume that $K\subset (0,1)^{d}$ and redefine $K_{\chi}$ as
\begin{equation*}
K_{\chi}=\bigcup_{x \in \chi \cap K} \upsilon_{\chi}(x) \cap [0,1]^{d},
\end{equation*}
in order to avoid trivial complications due to possibly infinite cells. Penrose \cite{Pen07} proved the remarkable fact that for $\chi=\X_{n}$
\begin{align*}
\E \varphi(\chi) \to & \enspace \Vol(K), \\
\E (\varphi_{\text{Per}}(\chi)) \to & \enspace 0,
\end{align*}
almost surely, with no need for assumptions on $K$'s shape.

\paragraph{}
To further assess the quality of the approximation with binomial input, we must quantify the previous convergence. The unbiasedness of the Poisson case does not occur with binomial input, mainly because of edge effects. Nevertheless those effects seem to decrease exponentially with the distance, like is customary for Voronoi cells.  The following result shows that the bias of the estimator $\varphi (\chi _{n})$ decreases geometrically with $n$, therefore it is negligible with respect to the standard deviation, as shown in the following sections. Also, it still holds when $(0,1)^{d}$ is replaced by an arbitrary set $U$ containing $K$  in its interior.

\begin{theorem}
\label{th:voronoi-bias}
Assume that $K$ is a compact set with positive volume and  let $U$ be an open set containing $K$. Let $\chi _{n}=\{X_{i},1\leqslant i\leqslant n\}$ be iid uniform variables on $U$. Then there is a constant $0<c<1$ depending only on $K$ and $d$ such that for $n\geqslant 1,$
\begin{align*}
\left| \E\Vol(K_{\chi _{n}})-\Vol(K) \right|\leqslant c^{n} .
\end{align*}
\end{theorem}
\begin{proof}
Let
$\chi_{k}=\{X_{i},i\leqslant k\}$. By homogeneity of the problem we can suppose $\Vol(U)=1$. The Voronoi approximation $K_{\chi_{n}}$ of $K$ satisfies
\begin{align}
\notag
\E(\Vol(K_{\chi_{n}})) = & \enspace \sum_{i=1}^{n} \E(1_{X_{i}\in K}\Vol(v_{\chi_{n}}(X_{i})\cap U))\\
\notag= & \enspace n \E(1_{X_{n}\in K}\Vol(v_{\chi_{n-1}}(X_{n})\cap U)) \\
\label{eq:exp-vol-K-chi-n}= & \enspace n \int_{K} \E\Vol(v_{\chi_{n-1}}(x)\cap U) dx.
\end{align}
Take $0<r<\frac{1}{2}d(K,U^{c})$. We have for all $x\in K$
\begin{align*}
\E(\Vol(v_{\chi_{n-1}}(x)\cap U))= & \enspace \E(\int_{U}1_{y\in v_{\chi_{n-1}}(x)})dy \\
= & \enspace \E(\int_{U} 1_{B(y,\|y-x\|)\cap \chi_{n-1}=\emptyset})dy \\
=& \enspace \int_{U} \P({B(y,\|y-x\|)\cap \chi_{n-1}=\emptyset})dy \\
=& \enspace \int_{U} (1-\Vol(B(y,\|y-x\|)\cap U))^{n-1}dy \\
= & \enspace \int_{B(x,r)}  (1-\kappa_{d}\|y-x\|^{d})^{n-1} dy+c_{n}
\end{align*}
where
\begin{align*}
c_{n}=\int_{U\setminus B(x,r)} (1-\Vol(B(y,\|y-x\|)\cap U))^{n-1} dy.
\end{align*}
For $y\in U\setminus B(x,r)$, let $B_{y}$ be the ball interiorly tangent to $B(y,\|y-x\|)$ with center on $[x,y]$  {and radius $r$. We have $B_{y}\subset B(y,\|y-x\|)$ by construction} and $B_{y}\subset U$ because $B_{y}\subset B(x,2r)$. It follows that 
\begin{align*}
c_{n}\leqslant \int_{U\setminus B(x,r)}(1-\Vol(B_{y}))^{n-1}dy=\int_{U\setminus B(x,r)}(1-\kappa_{d}r^{d})^{n-1}dy\leqslant c_{0}^{n}
\end{align*}for some $0<c_{0}<1$, noticing that $\kappa _{d}r^{d} <\Vol(U)\leqslant 1$ because $B(x,r)\subset U$.
From there, a polar change of coordinates yields
\begin{align*}
\E(\Vol(v_{\chi_{n-1}}(x)\cap U))  =&\int_{0}^{r}d\kappa_{d}t^{d-1}(1-\kappa_{d}t^{d})^{n-1} dt + c_{n}\quad \text{(because a $d$-sphere has surface $d\kappa_{d}$)} \\
=& \enspace \left[ -\frac{(1-\kappa_{d} t^{d})^{n}}{n}\right]_{0}^{r} + c_{n} \\
=& \enspace \frac{1}{n} +O(c^{n})
\end{align*}
for some $c\in (0,1)$. Reporting in \eqref{eq:exp-vol-K-chi-n} yields the result.
\end{proof}


Recalling that the estimator is unbiased if the underlying sample is Poisson in $\mathbb{R}^{d}$, this pleads in favor of Voronoi approximation against other estimators \cite{DevWise,Rod07} where the bias is not known and does not seem to be negligeable.

\subsection{Asymptotic normality }
\label{sec:volume-approx}
This subsection is concerned with the results of \cite{LacPec15}, where it is shown that with binomial input, the volume approximation $K_{\chi}$ is asymptotically normal when the number of points tends to $\infty $. Variance asymptotics and upper bounds on the speed of convergence for the Kolmogorov distance are also given.

We begin by stating the boundary regularity condition necessary for these results to hold, which is related to the boundary densities studied in the previous section. As explained in the introduction, it can be seen as a weakened form of the standardness assumption. Define, for all $r > 0$, the boundary neighbourhoods
\begin{align*}
\partial K_{r}= & \enspace \partial K + B(0,r),  \\
\partial K_{r}^{-}= & \enspace  \partial K_{r} \cap K, \\
\partial K_{r}^{+}= & \enspace \partial K_{r} \cap K^{c}.
\end{align*}

\begin{assumption}[Boundary permeability condition]
\label{roll}
A set $K$ with no improper points satisfies the boundary permeability condition whenever
\begin{equation}
\label{roll1}
\underset{r > 0}{\liminf} \enspace \frac{1}{\Vol(\partial K_{r})}\left(\int_{\partial K_{r}^{+}} f_{r}^{2}(x) \mathop{dx} + \int_{\partial K_{r}^{-}} g_{r}^{2}(x)\mathop{dx}\right) > 0.
\end{equation}
\end{assumption}
The following proposition gives a more meaningful equivalent for Assumption~\ref{roll}.

\begin{proposition} \label{roll-bis} Assumption~\ref{roll} holds if and only if
\begin{equation}
\label{roll2}
\liminf\limits_{r > 0} \frac{1}{\Vol(\partial K_{r})\kappa_{d}r^{d}}\int_{K\times K^{c}} 1_{||x-y|| \leqslant r} \mathop{dx} \mathop{dy} > 0.
\end{equation}
\begin{proof}
Let us begin by establishing the relation between the expression of (\ref{roll2}) and $K$'s boundary densities. By Fubini's theorem
\begin{align*}
\int_{K\times K^{c}} \frac{ \mathbf{1}_{||x-y|| \leqslant r}}{\kappa_{d}r^{d}} \mathop{dx} \mathop{dy} = & \enspace \int_{K} \frac{\Vol(B(x,r)\cap K^{c})}{\Vol(B(x,r))} \mathop{dx} \\
= & \enspace \int_{K^{c}} \frac{\Vol(B(x,r)\cap K)}{\Vol(B(x,r))} \mathop{dx},
\end{align*}
which rewrites simply as
\begin{equation}
\int_{K^{c}} f_{r}= \frac{1}{\kappa_{d}r^{d}}\int_{K\times K^{c}} \mathbf{1}_{||x-y|| \leqslant r} \mathop{dx} \mathop{dy} = \int_{K} g_{r},
\end{equation}
by definition of boundary densities.

\paragraph{}
Consider the function $h_{r}=\mathbf{1}_{K}g_{r}+\mathbf{1}_{K^{c}}f_{r}$. We have $0 \leqslant h_{r} \leqslant 1$ and $h_{r}=0$ outside of $\partial K_{r}$. Applying the Cauchy-Schwarz inequality gives
\begin{equation*}
\int h_{r}^{2}\leqslant \int h_{r} \leqslant \sqrt{\Vol(\partial K_{r})} \sqrt{\int h_{r}^{2}}
\end{equation*}
which rewrites as
\begin{align*}
 \enspace \frac{1}{\Vol(\partial K_{r})}\int_{\partial K^{+}_{r}} f_{r}^{2} + \int_{\partial K^{-}_{r}}g_{r}^{2} \leqslant  \enspace \frac{1}{\Vol(\partial K_{r})}&\left(\frac{2}{\kappa_{d}r^{d}}\int_{K\times K^{c}} \mathbf{1}_{||x-y|| \leqslant r} \mathop{dx} \mathop{dy}\right)  \\
\leqslant &\enspace \sqrt{\frac{1}{\Vol(\partial K_{r})}\left( \int_{\partial K_{r}^{+}} f_{r}^{2} + \int_{\partial K_{r}^{-}}g_{r}^{2} \right)},
\end{align*}
so that clearly \eqref{roll1} and \eqref{roll2} are equivalent.
\end{proof}
\end{proposition}

\begin{remark}
If $K$ is bi-standard with constant $\varepsilon$ then (\ref{roll1}) holds as well with the left hand being greater than $\varepsilon^{2}$. Hence bi-standarness implies the boundary permeability condition.
\end{remark}

\begin{remark}
\label{comparable-sides-boundary}
Note that
\begin{equation*}
\Vol(\partial K_{r}^{+}) \geqslant \frac{1}{\kappa_{d}r^{d}}\int_{K\times K^{c}} 1_{||x-y|| \leqslant r} \mathop{dx} \mathop{dy},
\end{equation*}
so that $\Vol(\partial K_{r}^{+})\ll \Vol(\partial K_{r})$ prevents \eqref{roll2} from being satisfied. Of course, the same reasoning holds with $\partial K_{r}^{-}$ instead. In other words, it is necessary for the boundary permeability condition to be fulfilled that both sides of the boundary have comparable volumes.
\end{remark}

\paragraph{}
We reproduce below the result derived in \cite[Th. 6.1]{LacPec15} for Voronoi approximation, modified to measure the distance to the normal of the variable $\Vol (K_{\chi _{n}})-\Vol(K)$, instead of $\Vol(K_{\chi _{n}})-\E \Vol(K_{\chi _{n}})$ like in the original result. 
This subtlety is important for dealing with practical applications and obtaining confidence intervals for $\Vol(K)$. We deal with Kolmogorov distance, also  adapted to confidence intervals, and defined by

\begin{align*}
\dK(U,V):=\sup_{t\in \mathbb{R}}\left| \P(U\leqslant t)-\P(V\leqslant t) \right|,
\end{align*} for any random variables $U,V$.
	
\begin{theorem}
\label{LRP}
Let $K$ be a compact  subset of $(0,1)^{d}$. Assume that for some $s < d$ 
\begin{align}
\label{eq:minkowski}
0<\liminf_{r>0}r^{s-d} \Vol(\partial K_{r})\leqslant \limsup_{r>0}r^{s-d } \Vol(\partial K_{r})<\infty,
\end{align}and that $K$ satisfies the boundary permeability condition (Assumption \ref{roll}). Then  \begin{align}
\label{eq:var}
0<\liminf_{r>0} \frac{\Var(\Vol(K_{\chi _{n}}))}{n^{-2+s/d}}\leqslant \limsup_{r>0} \frac{\Var(\Vol(K_{\chi _{n}}))}{n^{-2+s/d}}< \infty,
\end{align}
and for all $\varepsilon > 0$ there is $C_{\varepsilon }>0$ such that for all $n\geqslant 1$
\begin{align}
\label{eq:tcl}
\dK\left(  \frac{\Vol(K_{\chi _{n}})- \Vol(K)}{\sqrt{\Var\left( \Vol(K_{\chi _{n}})\right) }},N\right) \leqslant C_{\varepsilon}n^{-s/2d}\log(n)^{4-s/d+\varepsilon},
\end{align}where $N$ is a standard Gaussian variable.
\end{theorem}

\begin{proof} This result is almost exactly  \cite[Th. 6.1]{LacPec15}, except that there it is proved that 
\begin{align}
\label{eq:tcl-2}
\dK\left(  \frac{\Vol(K_{\chi _{n}})- \E\Vol (K_{\chi _{n}})}{\sqrt{\Var\left( \Vol(K_{\chi _{n}})\right) }},N\right) \leqslant C_{\varepsilon} n^{-s/2d}\log(n)^{4-s/d+\varepsilon}.
\end{align}
To have a similar bound  involving $\Vol(K)$ instead of $\E \varphi (\Vol(K_{\chi _{n}}))$,
let us first remark that for $\delta \in \mathbb{R}$, a random variable $U$, and $V=U+\delta $, 
\begin{equation*}
\dK(V,N)\leqslant \dK(V,N+\delta )+\dK(N+\delta ,N)
\leqslant \dK(U,N)+(2\pi )^{-1/2} | \delta |, 
\end{equation*}
since $\dK(V,N+\delta)=\dK(U,N)$. It follows that 
\begin{align*}
\dK\left(  \frac{\Vol(K_{\chi _{n}})- \Vol(K)}{\sqrt{\Var\left( \Vol(K_{\chi _{n}}) \right)}},N  \right) &\leqslant \dK\left(  \frac{\Vol(K_{\chi _{n}})- \E\Vol(K_{\chi _{n}})}{\sqrt{\Var\left( \Vol(K_{\chi _{n}}) \right)}},N  \right) \\
&\hspace{3cm}+(2\pi )^{-1/2}\left| \frac{\E \Vol(K_{\chi _{n}})-\Vol(K)}{\sqrt{\Var(\Vol(K_{\chi _{n}}))}} \right| \\
&\leqslant  \dK\left(  \frac{\Vol(K_{\chi _{n}})- \E\Vol(K_{\chi _{n}})}{\sqrt{\Var\left( \Vol(K_{\chi _{n}}) \right)}},N  \right) +\frac{O(c^{n})}{n^{-1+s/2d }}
\end{align*}for some $c\in (0,1)$ by Theorem~\ref{th:voronoi-bias}. Reporting the bounds of \eqref{eq:tcl-2} yields \eqref{eq:tcl}.

\end{proof}

\begin{remark}
The fact that $[0,1]^{d}$ is the support of the random sampling variables does not seem to have a great importance. Uniformity over $[0,1]^{d}$ eases certain estimates in the proof of \cite[Th. 6.1]{LacPec15} related to stationarity, but is not essential. If the variables are only assumed to have a positive continuous density $\kappa (x)>0$ on an open neighborhood of $\partial K$, it should be enough for similar results to hold. See Theorem \ref{th:voronoi-bias}, or \cite{Pen07}, for rigourous results in this direction. 
\end{remark}

\begin{remark}If $K$ satisfies all the hypotheses of Theorem \ref{LRP} except the boundary permeability condition, then we have \begin{align}
\sup_{t\in \mathbb{R}}&\left | \P\left(  \frac{\Vol(K_{\chi _{n}})-\E \Vol(K_{\chi _{n}})}{\sqrt{\Var\left( \Vol(K_{\chi _{n}}) \right)}} \leqslant t \right)-\P(N\leqslant t)\right |\notag\\
\label{eq:tcl-variance}&\hspace{4cm} \leqslant C_{\varepsilon}n^{\varepsilon }(\sigma ^{-2}n^{-2+s/2d}+\sigma ^{-3}n^{-3+s/d}+\sigma ^{-4}n^{-4+s/d})
\end{align}where $\sigma ^{2}$ is the variance of $\Vol(K_{\chi _{n}})$. See \cite[Th. 6.2]{LacPec15} for more details.
\end{remark}

\begin{remark}Set-estimation literature is also concerned with perimeter approximation \cite[Sec. 11.2.1]{KenMol}. In the context of Voronoi approximation, the study of the functional $\Vol(K_{\chi _{n}}\Delta K)$ has been initiated in  \cite{HevRei, ReiSpoZap}. Although the result is not formally stated, a bound of the form \eqref{eq:tcl-variance} for this functional is available using the exact same method. One has to work separately to obtain a variance lower bound. Such a result with Poisson input has been derived very recently in the paper of Yukich \cite{Yuk15}.

 {Results regarding the volume of the symmetric difference between the set and its approximation can be used to compare Voronoi approximation with another estimator. Indeed, the bound in $n^{-1/d}$ given in \cite{HevRei} for $\E\Vol(K_{\chi _{n}}\Delta K)$ is better than the bound in $(nr_{n}^{d})^{-1/2}$ of \cite{BCP08}, who use the Devroye-Wise estimator with a smoothing parameter $r_{n} \ll n^{-1/(d+1)}$.}
\end{remark}

The consequences of Theorems \ref{theorem1} and \ref{LRP} for sets $K$ with self-similar boundary are immediate, condition (\ref{eq:minkowski}) automatically holds by Proposition \ref{prop1}.

 \begin{corollary}
\label{corollary1}
Let $K$ be a compact set such that $\partial K$ is a finite union of self-similar sets satisfying Assumption \ref{ass:1}. Then (\ref{eq:var}) and (\ref{eq:tcl}) hold.
\end{corollary}

This corollary applies  to the Von Koch flake with $s=\ln(4)/\ln(3)$ (Example \ref{flake}). The conclusions of Theorem \ref{LRP} also apply for instance to the Von Koch anti flake, where three Von Koch curves are sticked together like for building the flake, but here the curves are pointing inwards, and not outwards (Figure~\ref{VKantiflake}). Assumption \ref{ass:1} is not satisfied on the whole boundary, but it is within an open ball of $\mathbb{R}^{d}$ intersecting one and only one of the three curves, and having (\ref{eq:positive-density-boundary}) on a self-similar $E$ with same Minkowski dimension as $\partial K$ is actually enough for the boundary permeability condition to hold.

\begin{figure}[!h]
\centering
\includegraphics[scale=0.6]{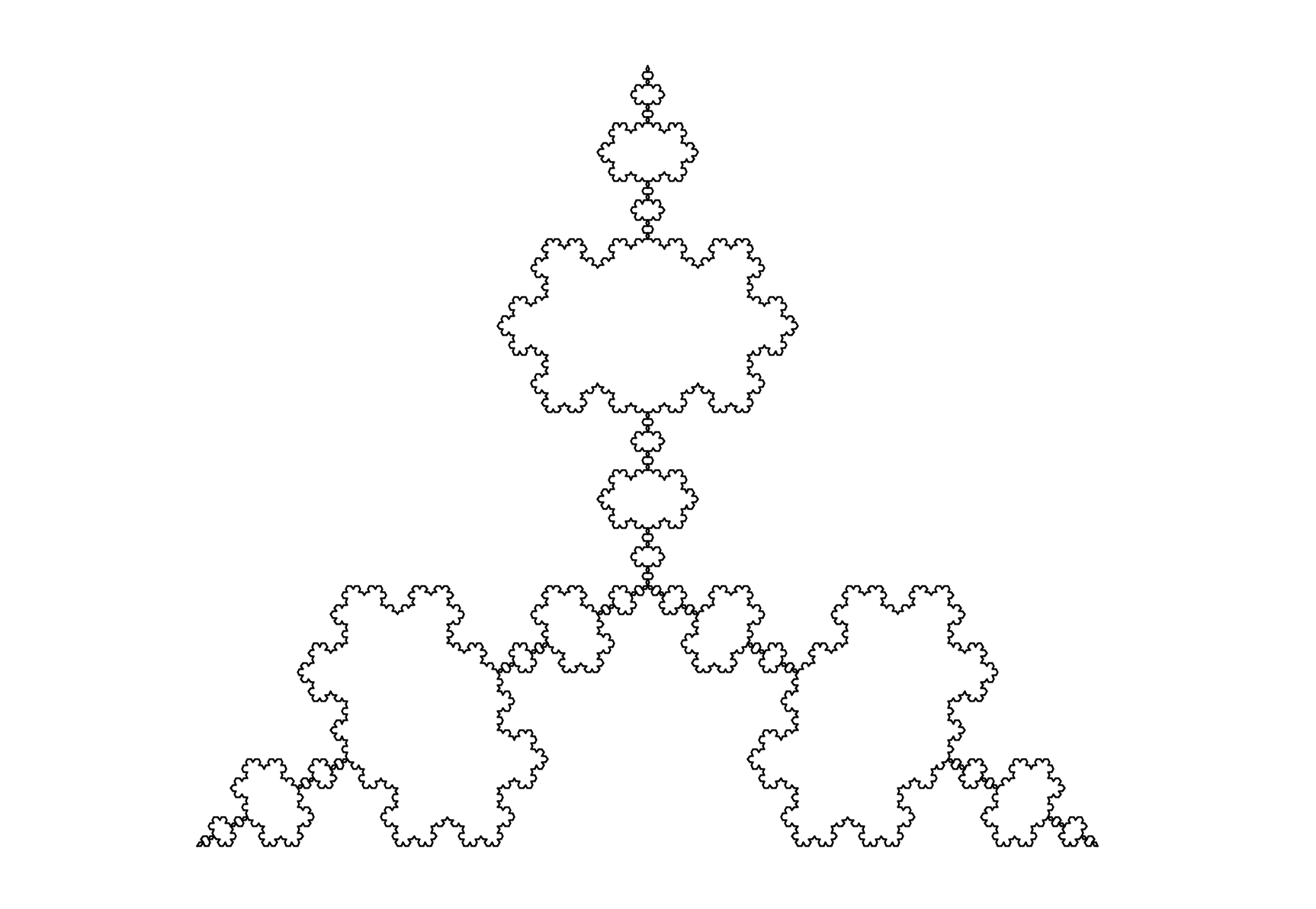}
\caption{\label{VKantiflake} The Von Koch antiflake}
\end{figure}

In Section \ref{sec:counter-example} we exhibit an example of a set $K$ such that $\partial K$ is self-similar and  $K$ does not satisfy Assumption \ref{ass:1}. We run simulations suggesting that (\ref{eq:var}) is also false. Our theorem gives a set of sufficient conditions, but other versions should be valid. For instance, the question of whether a compact set $K\subset \mathbb{R}^{2}$ whose boundary is a  locally self-similar Jordan curve  satisfies the conclusions of the theorem above seems to be of interest.

\subsection{Convergence for the Hausdorff distance}
\label{sec:hausdorff}
In this subsection we will make use of $r$-coverings and $r$-packings. Consider a collection $\mathcal{B}$ of balls having radius $r$ \emph{and centers belonging to some set} $E\subset \mathbb{R}^{d}$. $\mathcal{B}$ is said to be an $r$-packing of $E$ if the balls are disjoint. It is an $r$-covering if the balls cover $E$.

The size of minimal coverings and maximal packings is closely related to the Minkowski dimension of $E$. A necessary and sufficient condition for $E$ to have upper and lower Minkowski contents is that, for all small enough $r$, we can find an $r$-covering of $E$ with less than $Cr^{-s}$ balls, and an $r$-packing of the same set with more than $cr^{-s}$ balls. More related results can be found in \cite{Mattila}.

\paragraph{}
To estimate with precision $r$  the shape {of} a set $E$ by a point process $\chi$ it is often necessary to request that every point of $E$ is at distance less than $r$ of $\chi$. In the context of Voronoi approximation, this is made precise by the following lemma. Note that we only require $\chi$ to be dense enough near $\partial A$. This is, as suggested in the introduction, because Voronoi approximation fills in the interior regions of $K$ where points of $\chi$ are scarce.

\begin{lemma}
\label{lemma3}
Let $\chi \subset \mathbb{R}^{d}$ be a locally finite non-empty set.  
\begin{enumerate}
\item If every point $x$ of $\cl{\partial K_{r}^{+}}$ satisfies $d(x,\chi)< r$ then $K_{\chi} \subset K + B(0,r)$.
\item  If every point $x$ of $\cl{\partial K_{r}^{+}}$ satisfies $d(x,\chi)< r$ and every point $x$ of $\cl{\partial K_{r}^{-}}$ satisfies $d(x,\chi\cap K)< r$ then $d_{H}(K,K_{\chi})\leqslant r$.
\item If every point $x$ of $\cl{\partial K_{r}^{+}}$ satisfies $d(x,\chi\cap K^{c})< r$ and every point $x$ of $\cl{\partial K_{r}^{-}}$ satisfies $d(x,\chi\cap K)< r$  then $d_{H}(\partial K, \partial K_{\chi}) \leqslant r$.
\item If some point $x\in \partial K$ satisfies $d(x,\chi\cap K)\geqslant 3r$ and $d(x,\chi\cap K^{c})\leqslant r$ then $d_{H}(K,K_{\chi})\geqslant r$ and $d_{H}(\partial K,\partial K_{\chi})\geqslant r$ .
\end{enumerate}

\begin{proof}
We begin with the first point. Suppose $x\in K_{\chi}$ satisfies $d(x,K)\geqslant r$. Then there is a point $c_{x}\in \chi\cap K$ such that $x\in\upsilon_{\chi}(c_{x})$. The segment joining $c_{x}$ and $x$ contains points from $\partial K$ so we can consider $x_{0}$ the point of $\cl{\partial K_{r}^{+}}$ closest to $x$ on that segment. We have $d(x_{0},\partial K)=r$ and $x_{0} \in K^{c}$ since otherwise there would be another point of $\cl{\partial K_{r}^{+}}$ closer to $x$. As a consequence $d(x_{0},c_{x}) \geqslant r$. But then by assumption there is a point $y$ of $\chi$ such that $d(x_{0},y)< r$ and $c_{x}$ isn't the point of $\chi$ closest to $x$, which is a contradiction. Hence $x\in K_{\chi}$ implies $d(x,K)< r$ and $K_{\chi}\subset K+B(0,r)$.

\paragraph{}
Note that in the setting of points 2 and 3 we can apply the previous argument to $K^{c}$ instead of $K$, {the compacity of $K$ not playing any role in the proof. Along with $(K_{\chi})^{c}=(K^{c})_{\chi}$} this yields $K^{c}_{\chi }  \subset K^{c}+B(0,r)$, which reformulates as $K \setminus \cl{\partial K_{r}^{-}} \subset K_{\chi}$ by taking complements. Hence in both cases we have the inclusions
\begin{equation*}
K_{\chi}\subset K+ B(0,r), \enspace K_{\chi}^{c}\subset K^{c}+B(0,r),
\end{equation*}
and their reformulations
\begin{equation*}
K\setminus \cl{\partial K_{r}^{-}} \subset K_{\chi}, \enspace K^{c}\setminus \cl{\partial K_{r}^{+}} \subset K^{c}_{\chi}.
\end{equation*}

\paragraph{}
To prove the second point it is enough to show that $K \subset K_{\chi} + B(0,r)$. Let $x$ be a point of $K$. If $x\in K\setminus  \cl{\partial K_{r}^{-}}$ then $x$ belongs to $K_{\chi}$. And if $x$ is in $\cl{\partial K_{r}^{-}}$ then there is a point $y$ of $\chi \cap K$ such that $d(x,y)<r$. In all cases $x\in K_{\chi} + B(0,r)$.

\paragraph{}

We move on to point 3. The two inclusions $K^{c}\setminus \cl{\partial K_{r}} \subset K_{\chi}^{c}$ and $K\setminus \cl{\partial K_{r}} \subset K_{\chi}$ also show that if $x$ satisfies $d(x,\partial K) > r$, $x$ is interior to either $K_{\chi}$ or $K_{\chi}^{c}$. Hence $\partial K_{\chi} \subset \partial K+\cl{B(0,r)}$.  Conversely, for every point $x$ of $\partial K$ there are points of both $\chi \cap K$ and $\chi \cap K^{c}$ inside $B(x,r)$, so $B(x,r)$ contains a point of $\partial K_{\chi}$. Hence $\partial K \subset \partial K_{\chi}+B(0,r)$ and $d_{H}(\partial K, \partial K_{\chi})\leqslant r$.

\paragraph{} Lastly, suppose the requirements of point 4 are met. Let $y$ be a point of $\chi \cap B(x,r) \cap K^{c} $. Then all of the points in $B(x,r)$ are closer to $y$ than to the points outside of $B(x,3r)$. Consequently all points $B(x,r)$ must lie in Voronoi cells centered in $K^{c}$, and $x\notin K_{\chi}+B(0,r)$ so that $d_{H}(K,K_{\chi})\geqslant r$. The fact that $B(x,r)\subset K^{c}_{\chi}$ also implies $d(x,\partial K_{\chi}) \geqslant r$ and $d_{H}(\partial K,\partial K_{\chi})\geqslant r$.
\end{proof}

\end{lemma}

\paragraph{}
\label{essential2}

Now we apply this lemma to show almost sure convergence of $K_{\chi}$ in the sense of the Hausdorff distance.
To formulate such a result, the concept of proper points (beginning of Section \ref{sec:boundary-density}) proves to be useful. Improper points are invisible to the Voronoi approximation $K_{\chi}$ of $K$. Though this has no incidence when measuring volumes, it becomes a nuisance when measuring Hausdorff distances.

The set $K^{\text{prop}}$ of points proper to $K$ can be thought of as the complement of the biggest open set $O$ such that $\Vol(O\cap K)=0$, from which it follows that $K^{\text{prop}}$ is compact and that $K_{\chi}=K^{\text{prop}}_{\chi}$ a.s.

\begin{proposition}
\label{consistency-hausdorff}
$K_{\chi_{n}}\underset{n\rightarrow +\infty}{\longrightarrow}K^{\text{prop}}$ and $\partial K_{\chi_{n}}\underset{n\rightarrow +\infty}{\longrightarrow}\partial K^{\text{prop}}$ almost surely in the sense of the Hausdorff metric for both Poisson and binomial input.

\begin{proof}
Since $K_{\chi}=(K^{\text{prop}})_{\chi}$ almost surely and $K^{\text{prop}}$ has no improper points, this is equivalent to the fact that $K_{\chi_{n}}\rightarrow K$ and $\partial K_{\chi_{n}}\rightarrow \partial K$  almost surely when $K$ has no improper points. By the Borel-Cantelli lemma it is enough to show that both series
\begin{equation*}
\sum_{n\geqslant 1}\P(d_{H}(K_{\chi_{n}},K) > r), \enspace \sum_{n\geqslant 1}\P(d_{H}(\partial K_{\chi_{n}},\partial K) > r)
\end{equation*} are convergent for any positive $r$.

Consider $r/2$-coverings $\mathcal{B}^{+},\mathcal{B}^{-}$ of $\cl{\partial K_{r}^{+}}, \cl{\partial K_{r}^{-}}$ respectively. Since both sets are compact, these coverings can be made with finitely many balls. Set $\mathcal{B}=\mathcal{B}^{+}\cup \mathcal{B}^{-}$ and
\begin{equation*}
V=\min\left(\min_{B\in\mathcal{B^{-}}}\Vol(B\cap K),\min_{B\in\mathcal{B^{+}}}\Vol(B\cap K^{c})\right).
\end{equation*}
Because $K$ and $K^{c}$ have no improper points, $V>0$. If every ball of $\mathcal{B}^{+}$ contains a point of $\chi\cap K^{c}$ and every ball of $\mathcal{B}^{-}$ a point of $\chi\cap K$, then the requirements of points 2 and 3 in Lemma~\ref{lemma3} are met. The probability of this not happening is bounded by $|\mathcal{B}|(1-V)^{n}$ for binomial input and $|\mathcal{B}|e^{-nV}$ for Poisson input. In all cases the series associated with $\P(d_{H}(K_{\chi_{n}},K) > r)$ and $\P(d_{H}(\partial K_{\chi_{n}},\partial K) > r)$ converge, as required.
\end{proof}

\end{proposition}

A refinement of the method above gives an order of magnitude for $d_{H}(K,K_{\chi})$ with Poisson input, under assumptions on $\partial K, f_{r}$ and $g_{r}$ resembling those of Theorem~\ref{LRP}. This requires better estimations of the probability of the points of Lemma~\ref{covering-lemma} being met, which is the purpose of the following lemma.

\begin{lemma}
\label{covering-lemma}
Let $A,\chi$ be non-empty sets, and $\mathcal{B}$ a collection of balls centered on $A$ with radii $r$. Write $\mathcal{B}_{\tau}$ {for} the collection of balls having same centers as those of $\mathcal{B}$ but radius $\tau r$, and choose $\tau_{1},\tau_{2}>0$ such that $\tau_{1}+\tau_{2}=1$. If $\mathcal{B}_{\tau_{1}}$ is a $\tau_{1} r$-covering of $A$ and every ball of $\mathcal{B}_{\tau_{2}}$ contains a point of $\chi$, then $A \subset \chi + B(0,r)$.
\begin{proof}
Let $x$ be a point of $A$. By hypothesis, there is a ball of $\mathcal{B}$ with center $c$ such that $d(x,c)< \tau_{1} r$, and also a point $y$ of $\chi$ such that $d(y,c)< \tau_{2}r$. Hence $d(x,y)< r (\tau_{1}+\tau_{2})$ and $d(x,\chi)< r$. So indeed every point of $A$ is at distance less than $r$ of $\chi$.
\end{proof}
\end{lemma}


This handy lemma is meant to give probability estimations of events of the type $A\subset \chi + B(0,r)$, which are useful outside the context of Voronoi approximation. Typically, $\chi$ is chosen to be a random point process, and the covering $\mathcal{B}_{\tau_{1}}$  is chosen deterministically with as few balls as possible, often $C\tau_{1}^{-d}r^{-d}$. Bounding the probability that a ball of $\mathcal{B}_{\tau_{2}}$ does not intersect $\chi$ then gives an upper bound of the form
\begin{equation*}
\P(A \nsubseteq \chi + B(0,r))\leqslant |\mathcal{B}|\max_{B\in \mathcal{B}_{\tau_{2}}} \P(B\cap \chi = \emptyset).
\end{equation*}
The estimations obtained in such applications are less sensible to the number of balls in $\mathcal{B}$ than to their size. Hence optimal results are obtained when $\tau_{1}$ is small.

For example, the reader may use Lemma~\ref{covering-lemma} to derive \cite[Th. 1]{CueRod04} and its counterpart for Poisson input, which are concerned with the order of magnitude of $d_{H}(K,K\cap\chi)$ with $\chi$ an homogenous point process. Note that use of Minkowski contents and boundary densities give slighlty better bounds, which turn out to be optimal, see Remark~\ref{rem:points-vs-voronoi}.

\begin{theorem}
\label{theorem-haus}
Suppose that $\partial K$ has Minkowski dimension $s > 0$ with upper and lower contents, and that for all $r$ small enough, 
\begin{align*}
f_{r} \geqslant & \enspace \varepsilon \enspace \textnormal{on $K$}, \\
g_{r}\geqslant & \enspace \varepsilon \enspace \textnormal{on $K^{c}$}. 
\end{align*}
Then we have
\begin{align*}
&\P\left( \alpha \leqslant \frac{d_{H}(K,K_{\chi'_{\lambda}})}{(\lambda^{-1} \ln(\lambda))^{1/d}} \leqslant \beta \right) \underset{\lambda \rightarrow \infty}{\longrightarrow} 1 \\
&\P\left( \alpha \leqslant \frac{d_{H}(\partial K, \partial K_{\chi'_{\lambda}})}{(\lambda^{-1} \ln(\lambda))^{1/d}} \leqslant \beta \right) \underset{\lambda \rightarrow \infty}{\longrightarrow}  1
\end{align*}
where $\chi'_{\lambda}$ is a Poisson point process of intensity $\lambda$ and $\alpha,\beta$ satisfy $\alpha < \alpha_{K}, \beta > \beta_{K}$ with
\begin{align*}
\alpha_{K} = & \enspace\frac{1}{3}\left(\frac{s}{d\kappa_{d}(1-\varepsilon)}\right)^{1/d}, \\
\beta_{K} = & \enspace \left(\frac{s}{d\kappa_{d}\varepsilon}\right)^{1/d}.
\end{align*}
\end{theorem}

\begin{proof}
The approach of the proof is to tune $r$ in Lemma~\ref{lemma3} in order to have the events of points 3 happen with high probability. We shall only show the assertions regarding $d_{H}(\partial K,\partial K_{\chi})$, since the exact same arguments hold with $d_{H}(K, K_{\chi})$ as well.

We start with the upper bound. For all $\lambda$ let $\Omega_{\lambda}$ be the event where all the requirements from point 3 of Lemma~\ref{lemma3} are met with $\chi=\chi'_{\lambda}$, $r=r_{\lambda}=\beta(\lambda^{-1}\ln(\lambda))^{1/d}$. Hence $\{d_{H}(\partial K, \partial K_{\chi})> r \} \subset \Omega_{\lambda}^{c}$. We shall show that $\P(\Omega_{\lambda}^{c})\rightarrow 0$.

Choose $\tau_{1},\tau_{2}<1$ so that $\tau_{1}+\tau_{2}=1$ and $\tau_{2}\beta > \beta_{K}$. Let $\mathcal{B}^{+}$ be a collection of balls with radius $r$ and centers on $\cl{\partial K^{+}_{r}}$. As in Lemma~\ref{covering-lemma}, call $\mathcal{B}^{+}_{\tau}$ the collection of balls with same centers as those of $\mathcal{B}^{+}$, but radius $\tau r$. Define $\mathcal{B}^{-},\mathcal{B}^{-}_{\tau}$ similarily and set $\mathcal{B}=\mathcal{B}^{+}\cup \mathcal{B}^{-}$.  Note that $\mathcal{B}$ depends on $\lambda$, but  $\tau_{1},\tau_{2}$ do not.

We can and do choose $\mathcal{B}^{+},\mathcal{B}^{-}$ so that $\mathcal{B}_{\tau_{1}}^{+},\mathcal{B}_{\tau_{1}}^{-}$ are coverings of $\cl{\partial K_{r}^{+}}$ and $\cl{\partial K_{r}^{-}}$ respectively, and $|\mathcal{B}|$ has less than $C\tau_{1}^{-d}r^{-s}=C\tau_{1}^{-d}(\lambda/\ln{\lambda})^{-s/d}$ balls. Indeed, consider $\tau_{1}r/2$-packings of $\partial K_{r}^{+}$ and $\partial K_{r}^{-}$, both optimal in the sense that no ball can be added without losing the packing property. Because of volume issues, the packings have less than $C\tau_{1}^{-d}r^{-s}$ balls, and because of the optimality assumption doubling the radii of the balls gives the desired $\tau_{1}r$-coverings.


The intersection of $K$ with a ball $B \in \mathcal{B}_{\tau_{2}}^{-}$ of center $x$ has volume exactly $\kappa_{d}(\tau_{2}r)^{d}f_{\tau_{2}r}(x)$. Because $f_{r}\geqslant \varepsilon$ for large enough $\lambda$  and $\tau_{2}\beta > \beta_{K}$ it follows that
\begin{equation*}
\P(B \cap \chi \cap K = \emptyset) \leqslant 
\exp(-\lambda\tau _{2}^{d}\varepsilon\kappa_{d} r^{d})=\lambda^{-s/d-\delta}
\end{equation*}
for some $\delta>0$. The same bound is valid for $\P(B\cap\chi\cap K^{c}=\emptyset), B\in \mathcal{B}_{\tau_{2}}^{+}$.

Applying Lemma~\ref{covering-lemma} twice with $A=\partial K_{r}^{+},\partial K_{r}^{-}$ successively gives
\begin{equation*}
\P(\Omega_{\lambda}^{c}) \leqslant |\mathcal{B}|\lambda^{-s/d-\delta} \leqslant C \tau_{1}^{-d}\ln(\lambda)^{s/d}\lambda^{-\delta}
\end{equation*}
so that, since $\tau_{1}$ is fixed, $\P(\Omega_{\lambda}^{c})\rightarrow 0$ as desired.

\paragraph{}
The proof for the lower bound is quite similar. Fix $\delta > 0$, and redefine $\Omega_{\lambda}$ to be the event where the requirements described in point 4 of Lemma~\ref{lemma3} are met for $\chi=\chi'_{\lambda}$, $r=r_{\lambda}=\alpha(\ln(\lambda)\lambda^{-1})^{1/d}$ with $\alpha < \alpha_{K}$. Again, we shall show $\P(\Omega_{\lambda}^{c})\rightarrow 0$. Let $\mathcal{B}=\mathcal{B}_{\lambda}$ be a $3r$-packing of $\partial K$. We can assume $|\mathcal{B}|\geqslant cr^{-s}$.

%
%

The probability of there being no points of $K\cap\chi_{\lambda}$ in a ball $B(x,3r)$ of $\mathcal{B}$ and at least one point of $K^{c}\cap\chi$ in $B(x,r)$ for a point $x$ in the boundary is exactly
\begin{equation*}
\exp(-\lambda \kappa_{d}(1-g_{3r}(x))3^{d}r^{d})\left(1-\exp(-\lambda \kappa _{d} g_{r}(x) r^{d})\right)
\end{equation*}
because $B(x,3r_{\lambda})\cap K^{c}$ and $B(x,r)\cap K$ are disjoint. So we have the following upper bound, for $\lambda$ big enough 
\begin{equation*}
\P(\Omega_{\lambda}^{c})\leqslant (1-e^{-\lambda \kappa_{d}(1-\varepsilon)3^{d}r^{d}}(1-e^{-\lambda\kappa_{d} r^{d} \varepsilon}))^{|\mathcal{B}|} .
\end{equation*}
We would like the right hand to go to $0$ with $\lambda$. Taking logarithms this is equivalent to
\begin{equation*}
|\mathcal{B}|\exp(-\lambda \kappa_{d}(1-\varepsilon)3^{d}r^{d})(1-\exp(-\lambda \kappa_{d} r^{d} \varepsilon)) \underset{\lambda\rightarrow+\infty}{\longrightarrow} +\infty.
\end{equation*}
Because $\exp(-\lambda \kappa_{d}(1-\varepsilon)3^{d}r^{d})=\lambda^{\delta-s/d}$ with $\delta > 0$, $\exp(-\lambda \kappa_{d}r^{d}\varepsilon)\rightarrow 0$ and $|\mathcal{B}|\geqslant c(\lambda/\ln(\lambda))^{s/d}$, it is indeed the case.

\end{proof}

The proof and the result call for some comments. Most of them are minor variants on the result which were not included in the proof for clarity's sake.

\begin{remark}
\label{rem:no-need-minko}
It is possible to dispose of the hypothesis that $\partial K$ has Minkowski upper and lower contents, by using instead the so-called upper and lower Minkowski dimension, which always exist, see \cite{Mattila}. In particular, we can always do the coverings in the proof with $Cr^{-d}$ balls, so the upper bound still holds after replacing $s$ by $d$ in the expression of $\beta_{K}$. This compares with the result given by Calka and Chenavier in \cite[Corollary 2]{CalChe13}. One can also show, using the fact that $K$ is bounded and has positive volume, that $\mathcal{H}^{d-1}(\partial K)> 0$ so that $s$ can be replaced by $d-1$ in the expression of $\alpha_{K}$. Hence a lower bound also holds with no assumption on $\partial K$'s geometry when $d\geqslant 2$.

{For the results concerned with $d_{H}(\partial K, \partial K_{\chi})$, this is a remarkable feature that to our knowledge no other estimators possess. For instance, in \cite{CueRod04} a so-called expandability condition is required to obtain similar rates with the Devroye-Wise estimator.}

\end{remark}

\begin{remark}
If $s = 0$ and $\partial K$ has Minkowski contents then actually $d=1$, $\partial K$ has a finite number of points, and $d_{H}(K,K_{\chi})$ has order $\lambda^{-1}$ in the sense that for $\lambda$ large enough
\begin{equation*}
\P(d_{H}(K,K_{\chi'_{\lambda}})\lambda > t)\leqslant 2|\partial K|\exp(-2\varepsilon t),
\end{equation*}
which is enough to guarantee the existence of moments of all orders for $d_{H}(K,K_{\chi})\lambda$. This is not true of other shape estimators, and is due to the fact that Voronoi approximation only requires $\chi$ to be dense near $\partial K$ and not on all of $K$. If we don't have Minkowski contents the situation might be more delicate.
\end{remark}

\begin{remark}
Better estimations of the $\P(\Omega_{\lambda}^{c})$ in the proof along with an application of the Borel-Cantelli lemma yield the almost sure convergence rates advertised in the introduction. Explicitly
\begin{equation*}
\alpha_{K} \leqslant \liminf\limits_{n\rightarrow +\infty} \frac{d_{H}(K,K_{\chi'_{n}})}{(n^{-1} \ln(n))^{1/d}} \leqslant \limsup\limits_{n\rightarrow +\infty} \frac{d_{H}(K,K_{\chi'_{n}})}{(n^{-1} \ln(n))^{1/d}} \leqslant \beta'_{K}
\end{equation*}
and similarily for $d_{H}(\partial K, \partial K_{\chi})$, with $\alpha_{K},\beta_{K}$ as in Theorem~\ref{theorem-haus} and $\beta'_{K}=(\beta_{K}^{d}+(1/\kappa_{d}\varepsilon))^{1/d}$.
\end{remark}

\begin{remark}
For binomial input, some minor changes in the proof give the same upper bound. It can't be done for the lower bound since we use the fact that $\chi \cap A, \chi\cap B$ are independent when $A$ and $B$ are disjoint and $\chi$ is a Poisson point process.
\end{remark}

\begin{remark}
\label{rem:points-vs-voronoi}
Using similar techniques as in the proof above it is possible to show that
\begin{equation*}
\frac{d_{H}(K,K\cap\chi'_{\lambda})}{(\lambda^{-1} \ln(\lambda))^{1/d}} \overset{\P}{\longrightarrow}  \left(\frac{2(d-1)}{d\kappa_{d}}\right)^{1/d}
\end{equation*}
if $K$ has no improper points and $\partial K$ is a $\mathcal{C}^{2}$ manifold. Theorem~\ref{theorem-haus} shows that, under the same assumptions, the above limit can be used as an upper bound for $d_{H}(K,K_{\chi_{\lambda}})(\lambda/\ln(\lambda))^{-1/d}$. Hence, as a shape estimator, $K_{\chi}$ is not worse than $\chi \cap K$. It would be interesting to know if it is  better in some sense, a question related to the optimality of the bounds in Theorem~\ref{theorem-haus}.
\end{remark}

\begin{remark}
Applying point 2 of Lemma~\ref{lemma3} instead of point 3 in the proof of the theorem yields a better result for $d_{H}(K,K_{\chi})$. Specifically if $f_{r}\geqslant \varepsilon_{f}$ on $K$ then
\begin{equation*}
\P\left(\frac{d_{H}(K,K_{\chi'_{\lambda}})}{(\lambda^{-1} \ln(\lambda))^{1/d}} \leqslant \beta \right) \underset{\lambda \rightarrow \infty}{\longrightarrow} 1 
\end{equation*}
whenever
\begin{equation*}
\beta >  \left(\frac{s}{d\kappa_{d}\varepsilon_{f}}\right)^{1/d}.
\end{equation*}
{Together with Remark~\ref{rem:no-need-minko} this shows that inner standardness is a sufficient assumption to have convergence rates for $d_{H}(K,K_{\chi})$.}
\end{remark}

\subsection{A counter-example}
\label{sec:counter-example}
Here we construct a set $K_{\text{cantor}}$ with self-similar boundary not satisfying the boundary permeability condition. This example shows that Theorem~\ref{theorem1} cannot be generalised by dropping Assumption~\ref{ass:1}, even if the conclusion is weakened.

The example $K$ below is uni-dimensional, but a counter-example in  dimension $2$ can be obtained by considering $K\times [0,1]$.

\begin{example}
\label{cantor-example}
Let $E\subset \mathbb{R}$ the self-similar set generated by the similarities $\phi_{1}:x\mapsto x/3$, $\phi_{2}:x\mapsto (2+x)/3$ who satisfy the open set condition with $U=(0,1)$. $E$ is in fact the Cantor set, and can be characterized as the set of points having a ternary expansion with no ones.

\paragraph{}
$K_{\text{cantor}}$ will be defined as the closure of open intervals of $[0,1]\setminus E$. The trick is to choose few intervals with quickly decreasing length, so that $f_{r}$ is small on most of $K_{\text{cantor}}$'s boundary, but to distribute them well so that $\partial K_{\text{cantor}} =E$.
\paragraph{}
To every positive integer $n$ associate the sequence $s'^{n}$ of its digits in base $2$ in reverse order and double the terms to get $s^{n}$. For example, since $6$ is $110$ in base $2$, $s^{6}=(0,2,2)$. This defines a bijection between $\mathbb{N}$ and the set of finite sequences of zeroes and twos ending in 2, with the additional property that $s^{n}$ always has length $l_{n}\leqslant n$. Now for all $n$ define 
\begin{align*}
a_{n}= \enspace &\frac{1}{3^{n+1}}+ \sum_{k\geqslant 1}\frac{s^{n}_{k}}{3^{k}}\\
b_{n}= \enspace &\frac{2}{3^{n+1}}+ \sum_{k\geqslant 1}\frac{s^{n}_{k}}{3^{k}}\\
A_{n}= \enspace &(a_{n},b_{n})
\end{align*}
We have the following ternary expansions 
\begin{align*}
a_{n} &=\enspace 0.s^{n}_{1}s^{n}_{2}...s^{n}_{l_{n}}000...01 \\
&=\enspace 0.s^{n}_{1}s^{n}_{2}...s^{n}_{l_{n}}000...0022222... \\
b_{n} &=\enspace 0.s^{n}_{1}s^{n}_{2}...s^{n}_{l_{n}}000...02 \\
\end{align*}
Now, set $K=\cl{\bigcup A_{n}}$. We claim that $K$ has no improper points, $\partial K = E$ and that $K$ does not satisfy the regularity condition of Theorem~\ref{LRP}. 
\begin{proof}
The first assertion is easy to prove. Being segments, the $A_{n}$ have no improper points to themselves, so $\bigcup A_{n} \subset K^{\text{prop}}$ and $K\subset K^{\text{prop}}$ by taking closures.

\paragraph{}
For the second assertion we need to show that $\partial K = K\setminus \bigcup A_{n}=\cl{\bigcup \{a_{n},b_{n}\}}$. 
We already have the obvious $\partial K\subset K\setminus \bigcup A_{n}$. Define 
\begin{align*}
a'_{n}= \enspace &\frac{1}{3^{n+1}} -\frac{2}{3^{l_{n}}} + \sum_{k\geqslant 1}\frac{s^{n}_{k}}{3^{k}}\\
b'_{n}= \enspace &\frac{2}{3^{n+1}} -\frac{2}{3^{l_{n}}} + \sum_{k\geqslant 1}\frac{s^{n}_{k}}{3^{k}}\\
A'_{n}= \enspace &(a'_{n},b'_{n})
\end{align*}
Since for all $n, s_{l_{n}}^{n}=2$, the corresponding ternary expansions are 
\begin{align*}
a'_{n} &=\enspace 0.s^{n}_{1}s^{n}_{2}...s^{n}_{l_{n}-1}000...01 \\
&=\enspace 0.s^{n}_{1}s^{n}_{2}...s^{n}_{l_{n}-1}000...0022222... \\
b'_{n} &=\enspace 0.s^{n}_{1}s^{n}_{2}...s^{n}_{l_{n}-1}000...02 \\
\end{align*} 
If $x \in  A_{i}\cap A'_{j}$ then every ternary expansion of $x$ has the same digits as the finite ternary expansions of $a_{i},a'_{j}$ up to the first 1, which is impossible. So $\bigcup A'_{n}$ is an open set disjoint from $\bigcup A_{n}$ and hence from $K$. Furthermore, $\bigcup A'_{n}$ is dense near the $a_{n}$, because for all $k,N \in \mathbb{N}^{*}$, we can find an $a'_{k'}$ whose ternary expansion has the same $N$ first digits as the non-terminating expansion of $a_{k}$, so that $d(a_{k},a'_{k'}) \leqslant 1/3^{N}$. A similar argument works for the $b_{n}$, so that the $a_{n},b_{n}$ belong to $\partial K$ and, since the latter is closed, $\cl{\bigcup \{a_{n},b_{n}\}} \subset \partial K$.

Finally, consider a point $x\in K\setminus  \bigcup A_{n}$. For all $r > 0$, $B(x,r)$ contains a point from an $A_{k}$, and since $x \notin A_{k}$, one of the two points $a_{k},b_{k}$ must also be in $B(x,r)$. Consequently, $x$ is also an accumulation point of $\bigcup \{a_{n},b_{n}\}$. We just proved that $K\setminus \bigcup A_{n} \subset \cl{\bigcup \{a_{n},b_{n}\}}$. Putting this together with the previous two inclusions we get the desired equality.

Since for all $x \in E, N\in \mathbb{N}^{*}$ we can find an $a_{k}$ with the same first $N$ digits as $x$ in base 3, the $a_{n}$ are dense in $E$ and $E \subset \partial K$. Conversely, $\partial K \subset E$, since the $a_{n}, b_{n}$ belong to $E$, who is closed.

\paragraph{}
For the last assertion, pick any $r >0$ and set $N=2\lceil -\log_{3}(r)\rceil$. Let $X$ be the union of the balls of radius $r$ centered on the endpoints of the $N$ first $A_{n}$. $X$ has area at most $-4r\log_{3}(r)$ and for any $x \in \partial K_{r} \setminus X$, $B(x,r)$ does not intersect the $A_{k}, k\leqslant N$. Since $\Vol(\partial K)=0$ 
\begin{equation*}
\Vol(K\setminus (A_{1}\cup A_{2}\ldots \cup A_{N})) = \Vol(\bigcup_{n> N}A_{n}) = \frac{1}{2.3^{N+1}} \leqslant r^{2}.
\end{equation*}
 {But $\Vol(\partial K_{r})$ has order $r^{1-\ln(2)/\ln(3)}$ and
\begin{equation*}
\Vol(\partial K_{r}^{-}) \leqslant \Vol(X)+ \Vol(K\setminus (A_{1}\cup A_{2}\ldots \cup A_{N})) \leqslant -4r\log_{3}(r) + r^{2}
\end{equation*}
so that $\Vol(\partial K_{r}^{-}) \ll \Vol(\partial K_{r})$. According to Remark~\ref{comparable-sides-boundary}, this prevents \eqref{roll2} from holding.}
\end{proof}
\end{example}

\paragraph{}
Simulations were made for the quality of the Voronoi volume approximation with this set $K$. The magnitude order of the empirical variance of $\Vol(K_{\chi_{n}})$ seems to be $n^{\tau}$ with $\tau\approx -1.8$, as shown in Figure~\ref{regression}. Looking at Theorem~\ref{LRP}, the approximation behaves as if the set had a ``nice'' fractal boundary of dimension $\approx 0.2$, whereas its real fractal dimension is $1-\ln(2)/\ln(3)\approx 0.37$.

\begin{figure}[!h]
\centering
\includegraphics[scale=0.4]{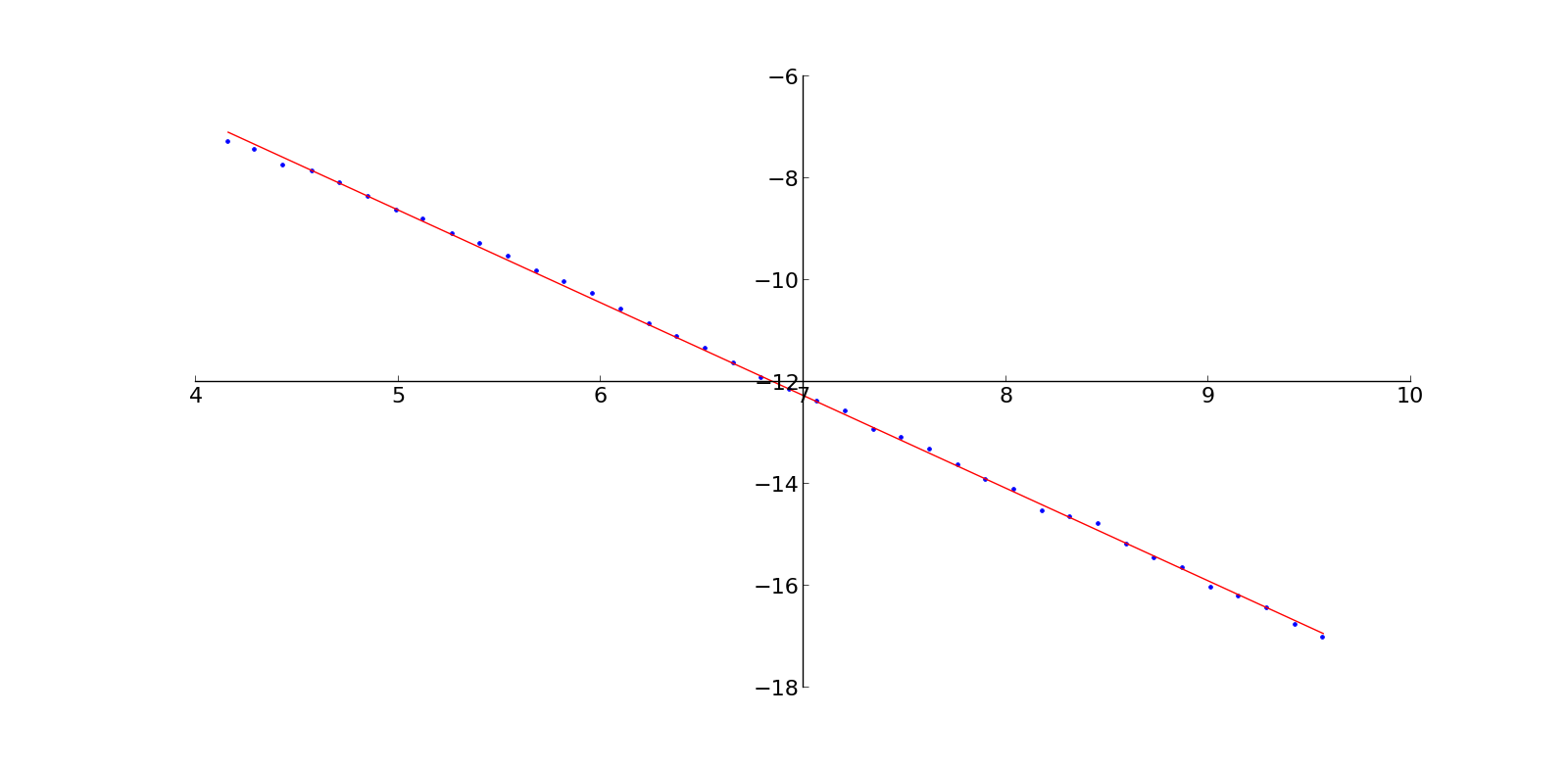}
\caption{\label{regression}In blue $\ln(\Var(K_{\chi_{n}}))$ as a function of $\ln(n)$, in red the associated linear regression. For each $n$, the variance was estimated with 1000 realisations of $\Vol(K_{\chi_{n}})$.}
\end{figure}

Simulations also suggest that a central limit theorem still holds. Such a fact indicates that though the results of Lachieze-Rey and Peccati \cite{LacPec15} seem to be generalisable, the variance of $\Vol(K_{\chi_{n}})$ is indeed related to the behaviour of $f_{r}$ and $g_{r}$ near $\partial K$.

\begin{example} 
It is possible to construct other sets not satisfying the regularity condition of Assumption~\ref{roll}. If we don't require $\partial K$ to be a self-similar set, a much simpler example is given by 
\begin{equation*}
K=\cl{\bigcup_{n\in \mathbb{N}^{*}} \left(\frac{1}{n}-\frac{1}{3^{n}},\frac{1}{n}\right)}.
\end{equation*}
Intentionally, $\partial K$ looks like the set $\{n^{-1}, n\in\mathbb{N}^{*}\}$, who is often given as an example of a countable set with positive Minkowski dimension. $K$ has no improper points, its boundary has Minkowski dimension $1/2$ with upper and lower contents, but $K$ does not satisfy (\ref{roll1}) or (\ref{roll2}). This can be proved using the same methods as in Example~\ref{cantor-example}. Again, simulations tend to show that the variance of $\Vol(K_{\chi_{n}})$ is about $n^{\tau  }$ with $\tau \approx -1,8$ and that a central limit theorem still holds.
\end{example}

\bibliographystyle{plain}
\bibliography{selfsimilar}

\end{document}